\documentclass[a4paper,12pt]{article}

\usepackage[english]{babel}
\usepackage[T1]{fontenc}
\usepackage{amssymb}
\usepackage{amsmath}
\usepackage{amsthm}
\usepackage{mathrsfs}
\usepackage[all]{xy}

\usepackage{sectsty}

\sectionfont{\fontsize{14}{15}\selectfont}

\theoremstyle{thmit} 
\newtheorem{thm}{Theorem}[section]
\newtheorem{lem}[thm]{Lemma}
\newtheorem{cor}[thm]{Corollary}
\newtheorem{prop}[thm]{Proposition}
\newtheorem{question}{Question}
\newtheorem{deff}[thm]{Definition}
\newcommand{\Z}{\mathbb{Z}}

\newcommand{\f}{\mathbb{F}}
\newcommand{\A}{{\cal A}}

\newcommand{\Oe}{\widehat{OE}}

\newcommand{\C}{{\cal C}}
\newcommand{\F}{{\cal F}}
\newcommand{\G}{{\cal G}} 
\newcommand{\W}{{\cal W}}

\newcommand{\K}{{\cal K}}

\newtheorem{exa}[thm]{Example}
\newtheorem{rem}[thm]{Remark}

\usepackage{amssymb}
\usepackage{amsmath}
\usepackage{amsthm}
\usepackage{mathrsfs}
\usepackage[all]{xy}
\usepackage{graphicx,color}

\title{Profinite genus of free products with finite amalgamation.}
\author{Vagner R. de Bessa, \ Anderson L. P. Porto and \ Pavel A. Zalesskii}

\begin{document}

\maketitle

\begin{abstract} 
A finitely generated residually finite group $G$ is an $\Oe$-group if any action of its profinite completion $\widehat G$  on a profinite tree with finite edge stabilizers admits a global fixed point. In this paper, we study the profinite genus of  free products $G_1*_HG_2$ of $\Oe$-groups $G_1,G_2$ with finite amalgamation $H$. Given such $G_1,G_2,H$ we give precise formulas for the number of isomorphism classes of $G_1*_HG_2$ and of its profinite completion. We compute the genus of $G_1*_HG_2$ and list various situations when the formula for the genus simplifies.
\end{abstract}


\section{Introduction}

There has been much recent study of whether residually finite groups, or classes of residually finite groups of combinatorial nature may be distinguished from each other by their sets of finite quotient groups.

In group theory the study in this direction started in 70-th of the last century when Baumslag \cite{BAU74},  Stebe \cite{Ste72} and others found examples of finitely generated residually finite groups having the same set of finite quotients. The general question  addressed in this study can be formulated as follows: 
		
		\begin{question}\label{question}
			To what extent a finitely generated residually finite group $\Gamma$ is  determined by its finite quotients?
		\end{question} 
		
		The study leaded to the notion of genus $\mathfrak{g}(G)$ of a residually finite group $G$, the set of isomorphism classes of finitely generated residually finite groups having the same set of finite quotients as $G$. Equivalently, $\mathfrak{g}(G)$ is the set of isomorphism classes of finitely generated residually finite groups having the  profinite completion isomorphic to the profinite completion $\widehat G$ of $G$. 
		
		The study mostly was concentrated to establish whether the cardinality $g(G)$ of  the  genus $\mathfrak{g}(G)$ is finite or  1 (see  \cite{GPS, GS, BCR16, BMRS18, BMRS20, J, M, Wil17} for example) (we use the same term {\it genus} for $g(G)$ from now on). There are only few papers where exact numbers or estimates of the genus appear due to the difficulties of such calculation (see \cite{GZ, BZ, Ner19, Ner20}).
However, the following 		
		 principle question of Remeslennikov is still open.
		
		\begin{question} (V.N. Remeslennikov)
		Is the genus $g(F)$ of a free group $F$ of finite rank equals 1?
		
		\end{question}

In \cite[Section 4]{GZ}  were given examples of non isomorphic free products with finite amalgamation whose profinite completion are isomorphic. This gives a motivation to investigate the profinite genus of such amalgamated free products and this is exactly the subject of our paper.

However, since Remeslennikov's question is not answered, it is not clear whether the profinite completion of an one-ended group does not split as a profinite amalgamated free product or an HNN-extension over a finite group. This naturally gives restriction to the family in which we make our considerations.

A finitely generated residually finite group $G$ with less than 2 ends, i.e. a finite or one ended group, will be called an $OE$-group in the paper.   It follows from the famous Stallings theorem that $G$ is an $OE$-group  if and only if whenever it acts on a tree with finite edge stabilizers
it  has a global fixed point. A finitely generated profinite group will be called $OE$-group if it has the same property:   whenever it acts   on a  profinite tree  with finite edge stabilizers, then it fixes a vertex.  Note that  finite and one-ended groups acting on a tree have a global fixed point and we need to extend this property to the profinite completion.  A finitely generated residually finite group $G$ will be called $\widehat{OE}$-group if  $\widehat G$ can not act on a profinite tree with finite edge stabilizers without a global fixed point (see Definition \ref{Profinite tree} for the definition of a profinite tree).
  Note that the class of $\widehat{OE}$-groups is quite large. It contains all finitely generated residually finite soluble groups, Fuchsian groups, indecomposable (into free products) 3-manifold groups, indecomposable RAAGs,  as well all arithmetic groups of rank $>1$. 
 
In this paper we study the genus of free products  of $\widehat{OE}$-groups with finite amalgamation. We divide the task into two parts depending on the point of view for free product with amalgamation.

One point of view is given two $\Oe$-groups $G_1,G_2$ and isomorphic finite subgroups $H_1\cong H_2$ of $G_1$ and $G_2$ respectively, we consider the family $\W$ of free products $G_1*_{H_1=H_2} G_2$ with $H_1$ and $H_2$ two fixed amalgamated copies of a finite group $H$. Given $G\in \W$ we first concentrate in calculation of the genus $g(G,\W)$ within this family of amalgamated free products (see Theorems \ref{fixed subgroups}, \ref{fixed subgroups prof non sym} and \ref{case symmetrical genus}).
 
From another point of view let $\F$ be the family of all amalgamated free products $G_1*_H G_2$ such that the embeddings of $H$ in $G_i \,(i=1,2)$ are not fixed. Given $G\in \F$ we give the precise formula for the genus $g(G,\F)$ (see Theorem \ref{genus amalgam}).


We first investigate the isomorphism problems for amalgamated free products and profinite amalgamated free products that are of independent interest and in fact calculate the number of isomorpism classes of such amalgamated free products.

We denote by  $Aut_{G_i}(H)$  the subgroup of automorphisms of $G_i$ that leave $H$ invariant and $\widetilde{Aut}_{G_i}(H)$ for the image of it in the group of outer automorphisms $Out(H)$ of $H$ (warning: the latter group depends on the embedding of $H$ in $G_i$).  

\begin{thm} \label{fixed images out intr} Let $G_1\not\cong G_2$  be $OE$-groups having fixed isomorphic copies $H_1$ and $H_2$ of a finite common  subgroup $H$, respectively.  Then the number of isomorphism classes of  amalgamated free products  $G_1*_{H_1=H_2}G_2$ is $$ |\widetilde{Aut}_{G_2}(H)\backslash Out(H)/ \widetilde{Aut}_{G_1}(H)|.$$ \end{thm}  


We continue to use tilde for the natural image in $Out(H)$. 
 Then using the standard notation $N_G(H)$ for the normalizer of $H$ in $G$, we can state the following  theorem establishing the formula for the genus.

\begin{thm}\label{fixed subgroups int} Let $G=G_1*_{H}G_2$ be a free product of $\widehat{OE}$-groups with finite amalgamation such that $\widehat G_1\not\cong \widehat G_2$.  Then the genus  
  $g(G,\W)$ equals
$$|\widetilde{Aut}_{G_2}(H)\backslash \widetilde{Aut}_{\widehat G_2}(H)\widetilde N^+\widetilde{Aut}_{\widehat G_1}(H)/\widetilde{Aut}_{G_1}(H)|,$$ where $N^+=\{n\in N_{\widehat G}(H)\mid \overline{\langle G_1, n^{-1} G_2 n\rangle}=\widehat G\}$.
\end{thm}

In particular we establish the list of situations when the formula for the genus is simpler.

\begin{thm}\label{genus 1} Let $G_1,G_2$ be $\widehat{OE}$-groups  and $H$ a finite common subgroup of $G_i$ such that $\widehat G_1\not\cong \widehat G_2$. Let $G=G_1*_HG_2$  be an   amalgamated free product. Then  $g(G,\W)$ equals
$$|\widetilde{Aut}_{G_2}(H)\backslash \widetilde{Aut}_{\widehat G_2}(H)\widetilde{Aut}_{\widehat G_1}(H)/\widetilde{Aut}_{G_1}(H)|,$$ if   one of the following conditions is satisfied:
\begin{itemize}
\item[(i)] $H$ is central in $G_i$ for $i=1$ or $2$;
\item[(ii)] $H$ is a direct factor of $N_{\widehat{G}_i}(H)$ for $i=1$ or $2$;
\item[(iii)]  $Out(H)$ is abelian (in particular if $H$ is cyclic; also the case for majority of simple groups);
\item[(iv)] $H$ is self-normalized in $\widehat{G}_i$;
\item[(v)] $H$ is a retract of $G_i$.

\end{itemize}
   
   \end{thm} 

If $G_1\cong G_2$ (resp. $\widehat G_1\cong \widehat G_2$)   we also prove the analogs of the theorems stated above; the formulas in this case are distorted by a possible automorphism that swaps $G_1$ and $G_2$ (resp. $\widehat G_1$ and  $\widehat G_2$). The hypothesis of $\widehat G_1\not\cong \widehat G_2$ in reality is a weaker condition of absence of an automorphism of $\widehat G$ that swaps $\widehat G_1$ and $\widehat G_2$. Situation when such automorphism exists is also considered.

To describe the genus in the family of push-outs $\F$ observe  that $Aut(G_i)$ acts naturally on the set $S_i$ of finite subgroups of $G_i$ and similarly  $Aut(\widehat G_i)$ acts naturally on the set of finite subgroups of $\widehat G_i$. Let $H$ be a common finite subgroup of $G_i\,(i=1,2).$ Denote by $Aut(G_i)H$ and $Aut(\widehat{G_i})H$ the orbits of $H$ with respect to these actions.



\begin{thm}\label{genus amalgam intr} Let $G_1, G_2$ be  $\Oe$-groups and $G=G_1 *_H G_2$ be a  free product with finite amalgamation. Then  the genus  $g(G,\mathcal{F})$ of $G$  equals to  $g(G,\F)=\displaystyle\sum_{\{H_1,H_2\}} g(G_1*_{H_1=H_2}G_2, \W)$, where $\{H_1,H_2\}$ ranges over unordered pairs of  representatives of $Aut(G_i)$-orbits of $S_i\cap Aut(\widehat G_i)H\,(i=1,2)$.

\end{thm}

Of course this formula gives finite number, only if the sets of  $Aut(G_i)$-orbits of $S_i\cap Aut(\widehat G_i)H\,(i=1,2)$ are finite and the genera of summands are finite. If    $G_i$  has finitely many conjugacy classes of finite subgroups then clearly these sets of orbits are finite. Arithmetic groups, mapping class groups, the automorphism group and the outer automorphism group of a free group of finite rank,  hyperbolic groups  and many groups of geometric nature share this property.  Note that arithmetic groups with congruence subgroup property have finite genus (see \cite{Aka}), so for such arithmetic groups $G_1,G_2$ the genus $g(G,\F)$ is finite.

Note that by \cite[Theorem 1.1]{BPZ} the above results can be used to compute  the genus  within a larger class $\A$ of accessible groups whose vertex groups of its JSJ-decomposition are $\widehat{OE}$-groups (see Remark \ref{class A}).


The structure of the  paper is as follows. Section 2 contains elements of the profinite version of the Bass-Serre theory used in the paper (see \cite{R} for more details) and some results on normalizers of amalgamated free (profinite) products.   In Section 3 we discuss properties of $\widehat{OE}$-groups. Sections 4 and 5 devoted to the isomorphism problem of amalgamated free  products and profinite  amalgamated free products which are of independent interest. We prove Theorems \ref{fixed subgroups int} and \ref{genus 1} in Section 6, where we give several examples of calculation of the genus of free products with finite 
 amalgamation.  Section 7 is dedicated to the proof of Theorem \ref{genus amalgam intr}, where many examples of  amalgamated free products $G=G_1*_H G_2$ with $g(G,\F)=1$ are given.


Our basic reference for notations and results about profinite groups is \cite{RZ}. We refer the reader to Lyndon-Schupp \cite{LS}, Magnus-Karrass-Solitar \cite{MKS}, Serre \cite{S} or Dicks-Dunwoody \cite{DI} for an account of basic facts on amalgamated free products  and to \cite{RZ} for the profinite versions of these constructions. Our methods based on the profinite version of the Bass-Serre theory of groups acting on trees that can be found in \cite{R}. All homomorphisms of profinite groups are assumed to be continuous in this paper. We will use the standard abbreviation $x^{g}$ for $g^{-1}xg$ when $g,x$ are elements of a group $G;$ the inner automorphism of $G$ corresponding to the conjugation $gxg^{-1}$ will be denoted by $\tau_g$ (as all maps are written on the left). The composition of  two maps $f$ and $g$ is often defined simply as $f\circ g = fg.$ For a subgroup $H$ of $G$ we shall denote by $Aut_G(H)$ the group of automorphisms that leave $H$ invariant and by $\overline{Aut}_G(H)$ its image in $Aut(H)$. All amalgamated free products $G=G_1*_HG_2$ (resp. profinite amalgamated free products $G=G_1\amalg_HG_2$) will be assumed non-fictitious in the paper, i.e. $G_1\neq H\neq G_2$.


\section{\large Preliminary Results}

In this section we recall the necessary notions of the Bass-Serre theory for abstract and profinite graphs.
	
	\begin{deff}[Profinite graph]
		A (profinite) graph is a (profinite space) set $\Gamma$ with a distinguished closed nonempty subset $V(\Gamma)$ called the vertex set, $E(\Gamma)=\Gamma-V(\Gamma)$ the edge set and two (continuous) maps $d_0,d_1:\Gamma \rightarrow V(\Gamma)$ whose restrictions to $V(\Gamma)$ are the identity map $id_{V(\Gamma)}$. We refer to $d_0$ and $d_1$ as the incidence maps of the (profinite) graph $\Gamma$.  
	\end{deff}
	
	A morphism $\alpha:\Gamma \longrightarrow \Delta$ of profinite graphs is a continuous map with $\alpha d_i=d_i \alpha$ for $i=0,1$. By \cite[Proposition 2.1.4]{R} every profinite graph $\Gamma$ is an inverse limit of finite quotient graphs of $\Gamma$. A profinite graph $\Gamma$ is called connected if every its finite quotient graph is connected as a usual graph. 
	
	\begin{deff}\label{Profinite tree} Let $\Gamma$ be a profinite graph. Define $E^{*}(\Gamma)=\Gamma/V(\Gamma)$ to be the quotient space of $\Gamma$ (viewed as a profinite space) modulo the subspace of vertices $V(\Gamma)$. Consider the free profinite $\widehat\Z$-modules $[[\widehat\Z(E^{*}(\Gamma),*)]]$ and $[[\widehat\Z V(\Gamma)]]$ on the pointed profinite space $(E^{*}(\Gamma),*)$ and on the profinite space $V(\Gamma)$, respectively. Denote by $C(\Gamma,\widehat\Z)$ the chain complex

	$$	\xymatrix{ 0 \ar[r] & [[\widehat\Z(E^{*}(\Gamma),*)]] \ar[r]^d &[[\widehat\Z V(\Gamma)]] \ar[r]^{\varepsilon} & \widehat\Z \ar[r] & 0}$$ of free profinite $\widehat\Z$-modules and continuous $\widehat\Z$-homomorphisms $d$ and $\varepsilon$ determined by $\varepsilon(v)=1$, for every $v \in V(\Gamma)$, $d(\overline{e})=d_1(e)-d_0(e)$, where $\overline{e}$ is the image of an edge $e \in E(\Gamma)$ in the quotient space $E^{*}(\Gamma)$, and $d(*)=0$. 
		One says that $\Gamma$ is a profinite tree if the sequence $C(\Gamma,\widehat\Z)$ is exact. 
	\end{deff}	
	
	  If $v$ and $w$ are elements of a tree (respectively profinite tree) $T$, one denotes by $[v,w]$ the smallest subtree (respectively profinite subtree) of $T$ containing $v$ and $w$.


 


Associated with the profinite graph of profinite groups $({\cal G}, \Gamma)$ there is
a corresponding  {\em  standard profinite tree}  (or universal covering graph)
  $S=S(G)=\bigcup_{m\in \Gamma}
G/\G(m)$ (cf. \cite[Theorem 3.8]{ZM1}).  The vertices of
$S$ are those cosets of the form
$g\G(v)$, with $v\in V(\Gamma)$
and $g\in G$; its edges are the cosets of the form $g\G(e)$, with $e\in
E(\Gamma)$; and the incidence maps of $S$ are given by the formulas: $$d_0 (g\G(e))= g\G(d_0(e)); \quad  d_1(g\G(e))=gt_e\G(d_1(e)) \ \ 
(e\in E(\Gamma), t_e=1\hbox{ if }e\in D), $$
where $D$ is a maximal subtree of $\Gamma$. 

\medskip 
 If $G=G_1\amalg_HG_2$ is an amalgamated free profinite product then $G=\Pi_1(\G, \Gamma)$ with $\Gamma$ having two vertices and one edge, $G_1,G_2$ being vertex groups and $H$ being an edge group. In contrast to the abstract case, $G_1,G_2$ do not always embed into $G$, i.e. a  profinite free product with amalgamation is not always proper. Note that $G$ is always proper  when $H$ is finite (see Exercise $9.2.7$ item $(3)$ in \cite{RZ}) which will be assumed for the rest of the paper.
 
  When $G$ is proper, there is
a corresponding  {\em  standard profinite tree}  (or universal covering graph)
  (cf. \cite[Theorem 3.8]{ZM1}) $S(G)=G/G_1\cup G/G_2\cup G/H$ and  $d_0 (g\G(e))= g\G(d_0(e)); \quad  d_1(g\G(e))=g\G(d_1(e)) \ \ 
(e\in E(\Gamma)).$ There is a natural  continuous action of
 $G$ on $S(G)$, and clearly $ G\backslash S$ is an edge with two vertices. 







\begin{prop}\label{o421} 

\begin{enumerate}  

\item[(i)] Let $R = R_{1}*_{L}\,R_2$. Then $R_i\cap R_j^r\leq   L^{b}$ for some $b \in R_i$, whenever $i\neq j$  or $r\not\in R_i$, $(i,j\in \{1,2\}).$ Moreover $N_{R}(K)= R_i $ for any normal subgroup $K\triangleleft R_i$ not contained in $L$, $(i=1,2).$

\item[(ii)] Let $R = R_{1}\amalg_{L}\,R_2$ be a profinite amalgamated free product. Then $R_i\cap R_j^r\leq   L^{b}$ for some $b \in R_i$, whenever $i\neq j$  or $r\not\in R_i$, $(i,j\in \{1,2\}).$ Moreover $N_{R}(K)= R_i$ for any  normal subgroup $  K\triangleleft R_i$  not contained in $L$, $(i=1,2).$

\end{enumerate}

\end{prop}

\begin{proof} (ii) is Corollary 7.1.5 (b) in \cite{R} or Corollary 3.13 in \cite{ZM1}. The proof of (i) is the same using the classical Bass-Serre theory instead of the profinite one. The last part of the statement follows from the first taking $r\in N_{R}(K)$.

\end{proof}

In the following proposition we compute the normalizer of an amalgamated finite subgroup $H$ of an amalgamated free product in the discrete and profinite categories.

\begin{prop}\label{cp1} Let $G=G_1*_{H} G_2$ be an amalgamated free product  of two residually finite groups $G_1$, $G_2$ with common finite subgroup $H$. Then the following hold: \begin{itemize}

\item[(i)] $N_{G}(H)=N_{G_1}(H)*_{H} N_{G_2}(H)$,

\item[(ii)] $N_{\widehat{G}}(H)=N_{\widehat{G_1}}(H)\amalg_{H} N_{\widehat{G_2}}(H)$.

\end{itemize}

\end{prop}

\begin{proof} (i) The abstract version of the proof is analogous to the profinite version and will be omitted;  the abstract version of results quoted in the proof below can be found in \cite{S} or \cite{DI}.

\medskip
(ii)   Let $S(\widehat{G})$ be the standard profinite tree on which $\widehat{G}$ acts continuously. Denote by $S(\widehat{G})^{H}$ the subset of $S(\widehat{G})$ of  fixed points with respect to the action of $H$. Then $S(\widehat{G})^{H}$ is a profinite tree (Theorem 2.8 in \cite{ZM1} or Theorem 4.1.5 \cite{R}), also note that $S(\widehat{G})^{H} \neq \emptyset$ since  $H$ is a finite group (see Theorem 2.10 in \cite{ZM1} or Theorem 4.1.8 \cite{R}). Moreover $N_{\widehat{G}}(H)$ acts continuously on $S(\widehat{G})^{H}.$ Indeed, let $s \in S(\widehat{G})^{H}, g \in N_{\widehat{G}}(H)$ and $h \in H,$ then: $$h\cdot(g\cdot s)=(hg)\cdot s=(gh')\cdot s = g\cdot (h'\cdot s)=g\cdot s,$$ for some $h'\in H,$  whence follows that $g\cdot s \in S(\widehat{G})^{H},$ where $g\cdot s$ means the action of $g$ on $s$. 

Now the graph $S(\widehat{G})^{H}/N_{\widehat{G}}(H)$ has only one edge. To prove this it suffices to show that given $g\in \widehat{G}$ and $e\in E(S(\widehat{G})^{H})$, then $ge\in E(S(\widehat{G})^{H})$ implies that $g\in N_{\widehat{G}}(H).$ But this is obvious because both $H$ and $H^{g^{-1}}$ stabilize $ge$  which implies $H= H^{g^{-1}}$ as needed. Thus $S(\widehat{G})/\widehat{G} \cong S(\widehat{G})^{H}/N_{\widehat{G}}(H)$ is an isomorphism of graphs. Now, note that $S(\widehat{G})$ is connected and simply connected (see Theorem 3.4 in \cite{ZM} or Theorem 6.3.5 \cite{R}), thus  we can apply Proposition 4.4 in \cite{ZM} or Theorem 6.6.1 \cite{R} to deduce that $N_{\widehat{G}}(H)=N_{\widehat{G_1}}(H)\amalg_{H} N_{\widehat{G_2}}(H)$.
\end{proof}

\begin{cor}\label{8} If $G_1$, $G_2$ are polycyclic-by-finite, then $N_{G}(H)$ is dense in $N_{\widehat{G}}(H)$. Furthermore, $$\overline{N}_{\widehat{G}}(H)=\overline{N}_{G}(H)$$ where $\overline{N}_{G}(H)$ and $\overline{N}_{\widehat{G}}(H)$ are the natural images (by restriction) of $N_{G}(H)$ and $N_{\widehat{G}}(H)$ in the automorphism group $Aut(H),$ respectively. 

\end{cor}

\begin{proof} By Proposition \ref{cp1}  $$N_{\widehat{G}}(H)=N_{\widehat{G_1}}(H)\amalg_{H} N_{\widehat{G_2}}(H)$$ and so  $\left\langle N_{\widehat{G_1}}(H), N_{\widehat{G_2}}(H)\right\rangle$ is dense in $N_{\widehat{G}}(H)$. By Proposition 3.3 in \cite{RSZ}  $N_{G_{1}}(H)$ is dense in $N_{\widehat{G}_{1}}(H)$ and $N_{G_{2}}(H)$ is dense in $N_{\widehat{G}_{2}}(H)$, therefore  $N_{G}(H)$ is dense in $N_{\widehat{G}}(H)$. Thus $$\overline{N}_{\widehat{G}}(H)=\left\langle \overline{N}_{\widehat{G_1}}(H), \overline{N}_{\widehat{G_2}}(H)\right\rangle=\left\langle \overline{N}_{G_1}(H),\overline{N}_{G_2}(H)\right\rangle=\overline{N}_{G}(H).$$ 


\end{proof}

We conclude this section observing that the proof of Proposition \ref{cp1} is valid for any free pro-$\C$ product with finite amalgamation, where $\C$ is a class of finite groups closed for subgroups, quotient and extensions. We state it here for future references.

\begin{prop}\label{profinite normalizer} Let  $\C$ be a class of finite groups closed for subgroups, quotients and extensions and $G=G_1\amalg_H G_2$ be a free pro-$\C$ product with finite amalgamation. Then $N_G(H)=N_{G_1}(H)\amalg_H N_{G_2}(H)$.\end{prop}

\section{OE-groups}

A finitely generated residually finite group $G$ with less than 2 ends, i.e. a finite or one ended group, will be called an $OE$-group in the paper.   It follows from the famous Stallings theorem that $G$ is an $OE$-group  if and only if whenever it acts on a tree with finite edge stabilizers
it  has a global fixed point. A finitely generated profinite group will be called $OE$-group if it has the same property:   whenever it acts   on a  profinite tree  with finite edge stabilizers, then it fixes a vertex. 

 A finitely generated residually finite group $G$  will be called $\widehat{OE}$-group if $\widehat G$ is $OE$-group. Note that an $\widehat{OE}$-group is automatically $OE$-group, since if a residually finite group $G$ splits as an amalgamated free product or HNN-extension over a finite group then so does $\widehat G$.  

The next proposition gives a sufficient condition for $G$ to be an $\Oe$-group.

\begin{prop}\label{virtually cyclic}(\cite[Proposition 3.1]{BPZ}) Let $G$ be a finitely generated residually finite group such that $\widehat G$ does not have a non-abelian free pro-$p$ subgroup. If  $G$ is not virtually infinite cyclic, then $G$ is an $\Oe$-group.

\end{prop}

 The class of residually finite $\widehat{OE}$-groups is quite large. All arithmetic groups of $rank >1$ as well as indecomposable into a free product  Fuchsian groups, 3-manifold groups and Right Angled Artin groups and many others are $\Oe$-groups.

If $G$   satisfies an identity, then $\widehat G$ satisfies the same identity and so does not have non-abelian free  pro-$p$ subgroups, so it is an $OE$ and an $\Oe$-group by Proposition \ref{virtually cyclic} unless it is virtually infinite cyclic. All residullay finite $Fab$ groups and in particular just-infinite groups are also $\widehat{OE}$-groups.  

\bigskip

\begin{prop} \label{2a}(\cite[Proposition 3.2]{BPZ}) Let $G = G_{1}*_{H}\,G_2$ and $B = B_{1}*_{K}\,B_2$ be   amalgamated free products of groups  with finite amalgamation.  Suppose  $G_i, B_i ( i=1,2)$ are  finitely generated residually finite $OE$-groups.  If $G\cong B$ then there exist an isomorphism $\psi: G \longrightarrow B$ such that $\psi(H)= K$, $\psi(G_1)= B_1$, $\psi(G_2)= B_2$  (up to possibly interchanging $B_1$ and $B_2$ in $B$).

\end{prop}

\begin{prop} \label{compart1}(\cite[Proposition 3.3]{BPZ}) Let $G = G_{1}\amalg_{H}\,G_2$ and $B = B_{1}\amalg_{K}\,B_2$ be   profinite amalgamated free products of  profinite groups  with finite amalgamation. Suppose  $G_i,B_i\  (i=1,2)$ are OE-groups.   If $G\cong  B $ then there exist an isomorphism $\psi:G\rightarrow  B$ such that $\psi(H) = K$, $\psi(G_1) =B_{1}$ and $\psi(G_2) ={B_{2}}^{b}$ for some $b\in B$ (up to possibly interchanging $B_1$ and $B_2$ in $B$).

\end{prop}

\begin{rem} \label{finite-by-cyclic}(\cite[Remark 3.4]{BPZ})
\begin{enumerate}

\item[(i)]
Proposition \ref{2a}  also holds if we assume that  $G_i, B_i$ ($i=1,2)$ are semidirect products   $M_i\rtimes \Z$, $N_i\rtimes \Z$ with  $M_i,N_i$ finite such that $M_i\neq H$, $N_i\neq K$. 

\item[(ii)]
Similarly, Proposition \ref{compart1}  also holds if we assume that $G_i, B_i$ ($i=1,2)$ are semidirect products   $M_i\rtimes \widehat\Z$, $N_i\rtimes \widehat\Z$  with $M_i,N_i$ finite such that $M_i\neq H$, $N_i\neq K$.

\end{enumerate}

\end{rem}

\section{Isomorphism problems for abstract amalgamations} \label{section4}

  In this section we shall estimate the number of isomorphism classes of amalgamated free products $G_1*_HG_2$ of given $OE$-groups $G_1$ and $G_2$ with a finite common subgroup $H$. In fact we investigate two different points of view on $G_1*_HG_2$ separately: as a push-out and as an amalgamation of two given isomorphic finite subgroups of $G_1$ and $G_2$.

\subsection{Isomorphism problem for an abstract amalgamation as a push-out} \label{Isomorphism problem for an abstract amalgamation as a push-out}

Let $\lambda:H\longrightarrow G_1$ and $\mu:H\longrightarrow G_2$ be embeddings. Then the amalgamated free product $G=G_1*_{\lambda, H,\mu} G_2$ is defined by  presentation $G=\langle G_1,G_2\mid rel(G_1), rel(G_2), \lambda(h)=\mu(h), h\in H\rangle$. It can be thought as a push-out  $(\lambda,\mu)$ 
 $$\xymatrix{G_1\ar[r]&G_1*_{\lambda, H,\mu} G_2\\
            H\ar[r]^{\mu}\ar[u]^{\lambda} &G_2\ar[u]}$$

Let $(\eta,\nu)$ 
 $$\xymatrix{G_1\ar[r]&G_1 *_{\eta, H,\nu} G_2\\
            H\ar[r]^{\nu}\ar[u]^{\eta} &G_2\ar[u]}$$
be another push-out.
A morphism of push-out diagrams $(\eta,\nu)\longrightarrow (\lambda,\mu)$  is  a triple $(\beta_1,\beta_2,\alpha)$ of homomorphisms $\beta_1:G_1\longrightarrow G_1$, $\beta_2:G_2\longrightarrow G_2$, $\alpha:H\longrightarrow H$ such that 

\begin{equation}\label{morphism of pushouts}
\lambda\alpha=\beta_1\eta, \mu\alpha=\beta_2\nu;
\end{equation}
or equivalently 
\begin{equation}\label{pushout morphism}
\lambda^{-1}\beta_1\eta(h)=\mu^{-1}\beta_2\nu(h)=\alpha(h),
\end{equation} 
for each $h\in H$. Clearly, isomorphism of push-outs induces the isomorphism of the corresponding  amalgamated free products.

The next theorem is concerned with the isomorphism problem for  free products of $OE$-groups with finite amalgamation. 

\begin{thm}\label{6} 
Let  be $\lambda,\eta:H\longrightarrow G_1$, $\mu, \nu: H\longrightarrow G_2$ be monomorphisms of a finite group $H$  into $OE$-groups  $G_1$, $G_2$. Let  $G=G_1*_{\lambda, H,\mu} G_2$  and $B=G_1*_{\eta, H, \nu} G_2$ be  amalgamated free products.  Then $G\cong B$ if and only if (up to possibly interchanging  $G_1, G_2$ in case they are isomorphic) there exists an isomorphism of push-outs $(\beta_1,\beta_2,\alpha):(\eta,\nu)\longrightarrow (\lambda,\mu)$, $\beta_{1} \in Aut(G_1)$,  $\beta_{2} \in Aut(G_2)$ and $\alpha \in Aut(H)$.
\end{thm}


\begin{proof} $(\Leftarrow)$ Follows directly from the fact that  an isomorphism of push-outs yields an isomorphism of the amalgamated free products. 

\medskip
$(\Rightarrow)$ By Proposition \ref{2a} there is an isomorphism $\psi:B\longrightarrow G$ such that $\psi(G_i)=G_i, i=1,2$ (up to renumeration of $G_1$ and $G_2$ if $G_1\cong G_2$). Define $\beta_{i}=\psi|_{G_i}: G_i \longrightarrow G_i$. Then for all $h\in H$, we have
$\beta_{2}(\nu(h))=\psi(\nu(h))=\psi(\eta(h))=\beta_{1}\eta(h)=\lambda(\lambda^{-1}\beta_{1}\eta)(h)=\mu(\lambda^{-1}\beta_{1}\eta)(h)$, so $\lambda^{-1}\beta_1\eta=\mu^{-1}\beta_2\nu=\alpha$, for $\alpha\in Aut(H)$ as required.  
\end{proof}

 Now we shall calculate the number of isomorphism classes of $G_1*_HG_2$  for groups $G_1,G_2,H$ satisfying the hypothesis of this section. We shall split up our consideration into two cases: $G_1\not\cong G_2$ and $G_1\cong G_2$. 
 
Let  $Inj(H,G_i)$ be the set of monomorphisms of $H$ into $G_i\,(i=1,2)$. Then $Aut(H)$ acts on $Inj(H,G_i)$ naturally on the right by $\varphi \cdot\alpha=\varphi\alpha$, $\varphi\in Inj(H,G_i)$, $\alpha\in Aut(H)$. The automorphism group $Aut(G_i)$ acts naturally on $Inj(H,G_i)$ on the left $\beta \cdot \varphi=\beta\varphi$, $\beta\in Aut(G_i)$. 

\bigskip
Then  $Inj(H,G_1)\times Inj(H,G_2)$ gives the set  of push-outs 

$$\xymatrix{G_1\ar[r]&G_1*_H G_2\\
            H\ar[r]^{}\ar[u]^{} &G_2\ar[u]}$$ on which $Aut(G_1)\times Aut(G_2)$ acts on the left by the action given by $(\beta_1, \beta_2)\cdot (\lambda, \mu)=(\beta_1\lambda, \beta_2\mu)$ and we can define the diagonal action of $Aut(H)$ on it on the right  by $(\lambda,\mu)\alpha=(\lambda\alpha, \mu\alpha)$, $\alpha\in Aut(H)$, $\beta_i \in Aut(G_i)\,(i=1,2).$ 
            
\subsubsection{$G_1\not\cong G_2$}          

\bigskip
\begin{prop}\label{pushout orbit}            
            Suppose  that  $G_1\not\cong G_2$. Then push-out diagrams 
            $$\xymatrix{G_1\ar[r]&G_1 *_{\lambda, H, \mu} G_2\\
            H\ar[r]^{\mu}\ar[u]^{\lambda} &G_2\ar[u]} \hspace{3cm}
\xymatrix{G_1\ar[r]&G_1 *_{\eta,H,\nu} G_2\\
            H\ar[r]^{\nu}\ar[u]^{\eta} &G_2\ar[u]}$$ are in different $(Aut(G_1)\times Aut(G_2), Aut(H))$-orbits if and only if they are not isomorphic if an only if $G_1*_{\lambda, H, \mu}G_2\not\cong G_1*_{\eta, H, \nu}G_2$.\end{prop}
            
            \begin{proof}
            By Theorem \ref{6} $G_1*_{\lambda, H, \mu}G_2\cong G_1*_{\eta, H, \nu}G_2$ if and only if $(\lambda,\mu)=(\beta_1\eta\alpha, \beta_2\nu\alpha)$ for some $\beta_i\in Aut(G_i)$, $\alpha\in Aut(H)$ if and only if    $\lambda^{-1}\beta_1\eta(h)=\alpha^{-1}(h)=\mu^{-1}\beta_2\nu(h)$, for all $h \in H$ by equation (\ref{pushout morphism}).\end{proof}
     
     \medskip       
 Thus one deduces the following 
            
\begin{cor}\label{isomorphisms number} Let $H$ be a finite group and $G_1 \not \cong G_2$ be $OE$-groups. Then the number of isomorphism classes of amalgamated free products $G_1*_HG_2$   is $$|Aut(G_1)\times Aut(G_2)\backslash Inj(H,G_1)\times Inj(H,G_2)/Aut(H)|,$$ where $Aut(H)$ acts diagonally.
            
            \end{cor}

\subsubsection{ $G_1\cong G_2$}

Now we consider the case when $G_1\cong G_2$. Then we can think of $G_1, G_2$ as the same group $G$. 

            
            
           





 Consider now the set $Inj(H,G)$. Then as before $Aut(H)$ acts on the set of push-outs $(\lambda,\mu)$ diagonally and the wreath product $W=Aut(G) \wr C_2$ acts on the set of push-outs $(\lambda,\mu)$ on the left, with $C_2$ permuting $\lambda$ and $\mu$.
Hence Proposition \ref{pushout orbit} in this case reads as follows.

\begin{prop}\label{symmetric pushout orbits} Let $H$ be a finite subgroup of an $OE$-group $G$. Then push-outs $(\lambda,\mu)$ and $(\eta,\nu)$ are in different $(Aut(G)\wr C_2, Aut(H))$-orbits if and only if they are not isomorphic if and only if $G*_{\lambda, H, \mu}G\not\cong G*_{\eta, H, \nu}G$.
\end{prop}

 Corollary \ref{isomorphisms number} in this case reads as follows.

\begin{cor}\label{symmetrical isomorphisms number} Let $H$ be a finite subgroup of an $OE$-group $G$. Then the  number of isomorphism classes of amalgamated free products $G*_H\widetilde{G}$ (with $\widetilde{G}$ being a copy of $G$)  is $$|(Aut(G)\wr C_2)\backslash (Inj(H,G)\times Inj(H,G))/Aut(H)|,$$
 where $Aut(H)$ acts diagonally.
            
            \end{cor}

\begin{rem} \label{finite-by-cyclic isomorphism problems} The results of Subsection \ref{Isomorphism problem for an abstract amalgamation as a push-out} also hold if we assume that  $G_i$ \, ($i=1,2)$ are semidirect products $M_i\rtimes \Z$ with $M_i$ finite such that $M_i\neq H$.  
This follows from Remark \ref{finite-by-cyclic} (i).

\end{rem}

\subsection{Amalgamating given subgroups}\label{given subgroups}

In this subsection we consider the isomorphism problem of an amalgamated free product viewed as two given isomorphic subgroups $H_1,H_2$ of $G_1,G_2$ being amalgamated.
 But then in  Proposition \ref{pushout orbit} the images $H_1$ and $H_2$ of $H$ in $G_i\,(i=1,2)$ are fixed and so $Inj(H,G_1) \times Inj(H,G_2)$ are replaced by $Inj(H,H_1)\times Inj(H,H_2)\cong Aut(H)\times Aut(H)$ and $Aut(G_i)$ are replaced by $\overline{Aut}_{G_i}(H)$.  Note however, that  $\overline{Aut}_{G_i}(H)$ depends on the embedding  of $H$ into $G_i$. We identify   $H$ with $H_1$ in this section and denote by $\mu:H\longrightarrow H_2$ an isomorphism of $H$ to $H_2$. 
If $\beta_2 \in Aut_{G_2}(H_2)$ then $\overline{\beta_2}=\beta_2^{\mu} \in Aut(H)$ because  $\overline{\beta_2}(h)=(\mu^{-1} \beta_2 \mu)(h) \in H$ for each $h \in H$. Denote by $$\overline{Aut}^{\mu}_{G_2}(H)=\{\overline{\beta}=\beta^{\mu}| \beta \in Aut_{G_2}(H_2) \}$$ the image of $Aut_{G_2}(H_2)$ in $Aut(H)$. Then $\overline{Aut}^{\mu}_{G_2}(H)$ acts on $H$.
 
 Now by Propositions \ref{pushout orbit} and \ref{symmetric pushout orbits}  the diagonal action of $Aut(H)$ on $$Inj(H,H_1)\times Inj(H,H_2)\cong Aut(H)\times Aut(H)$$ does not change the amalgamated free product $G_1*_{H,\mu} G_2$ up to isomorphism.

 \begin{lem}\label{mod diagonal} There is a natural bijection from $$\overline{Aut}_{G_1}(H)\times \overline{Aut}^{\mu}_{G_2}(H)\backslash Aut(H)\times Aut(H)/Aut(H)$$ to the double cosets $\overline{Aut}^{\mu}_{G_2}(H)\backslash Aut(H)/\overline{Aut}_{G_1}(H)$.
 \end{lem}
 
\begin{proof}  Note that the factorization of $Aut(H)\times Aut(H)$ modulo the  diagonal $\Delta$ identifies $(h,1)$ with $(1,h^{-1})$. So  $(Aut(H)\times Aut(H))/\Delta$  with  $\overline{Aut}_{G_1}(H)\times \overline{Aut}^{\mu}_{G_2}(H)$ acting on the left on it can be viewed as $Aut(H)$ with $\overline{Aut}^{\mu}_{G_2}(H)$ acting on the left on it and $\overline{Aut}_{G_1}(H)$ acting on the right on it. The result follows.

\end{proof} 
 
\begin{lem}\label{troca} Let $\mu\in Inj(H, G_2)$and $\gamma:G_1\rightarrow G_2$ be an isomorphism  that sends $H$ to $\mu(H)$. Then for any $\alpha \in Aut(H)$ there exist  an isomorphism $\theta:G= G_1 *_{ H,\mu\alpha}G_2 \rightarrow B= G_1 *_{ H, \gamma(\mu\alpha)^{-1}\gamma} G_2$ that maps $G_1$ onto $G_2$ and $G_2$ onto $G_1$.
 \end{lem}
 \begin{proof} We show that  $\gamma$ and $\gamma^{-1}$ induce $\theta$ by the universal property. So it suffices to show that $(\gamma,\gamma^{-1},\gamma^{-1}\mu\alpha):(id,\mu\alpha)\longrightarrow (\gamma(\mu\alpha)^{-1}\gamma,id)$ is a morphism of push-outs.  In our case \eqref{pushout morphism} reads as follows: 
 $$(\gamma(\mu\alpha)^{-1}\gamma)^{-1}\gamma=\gamma^{-1}\mu\alpha$$
 that obviously holds. 
 \end{proof}

\begin{rem} \label{trocaprofinite} Let $G_1,G_2$ be profinite $OE$-groups, $\mu\in Inj(H, G_2)$ and  there is an isomorphism $\gamma:G_1\rightarrow G_2$  that sends $H$ to $\mu(H)$. Analogously to what was done in the previous lemma, it can be shown that there is an isomorphism $$\hat\theta: G_1 \amalg_{ H,\mu\alpha}G_2 \rightarrow G_1 \amalg_{ H, \gamma(\mu\alpha)^{-1}\gamma} G_2$$ for any $\alpha \in Aut(H)$, that maps $G_1$ onto $G_2$ and $G_2$ onto $G_1$. 
\end{rem}

\begin{prop}\label{gamma} Let $G=G_1 *_{ H, \mu_1}G_2$ and $B=G_1 *_{ H, \mu_2}G_2$ be free products of OE-groups with finite amalgamation, where $\mu_i\in Iso(H,H_2), i=1,2$.
\begin{enumerate}
\item[i)] Suppose there is no isomorphism $\gamma:G_1\rightarrow G_2$ such that $\gamma(H)=H_2$. Then $G\cong B$ if and only if $\mu_2=\beta_2 \mu_1\beta_1$ for some $\beta_1\in Aut_{G_1}(H)$, $\beta_2\in Aut_{G_2}(H_2)$.
\item[ii)] Suppose there is an isomorphism $\gamma:G_1\rightarrow G_2$ such that $\gamma(H)=H_2$. Then $G\cong B$ if and only if $\mu_2=\beta_2 \mu_1\beta_1$ or $\gamma\mu_1^{-1}\gamma=\beta_2 \mu_2\beta_1$ for some $\beta_1\in Aut_{G_1}(H)$, $\beta_2\in Aut_{G_2}(H_2)$.

\end{enumerate}
\end{prop} 
 
\begin{proof}
$(\Rightarrow)$ Let $\psi: B\rightarrow G$ an isomorphism. By Proposision \ref{2a}  we may assume that $\psi(G_i)=G_i\, (i=1,2)$ or $\psi(G_1)=G_2$ and $\psi(G_2)=G_1$ (note that the latter case occurs when $G_1\cong G_2$ and $\psi|_{G_1}:G_1\rightarrow G_2$ with $\psi(H)=H_2$). We divide the proof into two cases.

Case i) $\psi(G_i)=G_i, i=1,2$.  Put $\beta_1=\psi|_{G_1}$ and $\beta_2 = (\psi|_{G_2})^{-1}$. Since $\mu_2(h)=h$ and $\psi(h)=\mu_1(\psi(h))$ one has $(\beta_2^{-1}\mu_2)(h)=(\psi\mu_2)(h)=\psi(h)=\beta_1(h)=\mu_1\beta_1(h)$ for all $h\in H$. Therefore $\mu_2=\beta_2 \mu_1\beta_1$ where  $\beta_1\in Aut_{G_1}(H)$ and  $\beta_2\in Aut_{G_2}(H_2)$.\\

Case ii) If $\psi(G_1)=G_2$ and $\psi(G_2)=G_1$ we apply Lemma \ref{troca} with $\alpha=id$ and $\mu=\mu_1$; then we have an isomorphism $\theta:G\longrightarrow G_1 *_{ H, \gamma\mu_1^{-1}\gamma}G_2$. Put
 $\varphi=\theta\psi$. Then  $\varphi(G_i)=G_i$ $(i=1,2)$ and the result follows from Case i).

$(\Leftarrow)$ Follows directly from the fact that an isomorphism $(\beta_1,\beta_2):(id,\mu_2)\rightarrow (id, \mu_1)$ of push-outs
yields an isomorphism of the amalgamated free products.

\end{proof} 
 

Our goal now is to find formulas to calculate the number of isomorphism classes of the amalgamated free products of $OE$-groups $G_1$ and $G_2$ with finite amalgamation of fixed subgroups $H_1\leq G_1$, $H_2\leq G_2$. Thus interpreting $H_1,H_2$ as one common subgroup $H_1=H=H_2$ of $G=G_1 *_{H} G_2$ (i.e assuming w.l.o.g. $\mu_1=id$) 
 Corollary \ref{isomorphisms number} admits the following simple form in this case.
 
 \begin{cor}\label{fixed images isomorphisms number} Let    $G_1$ and $G_2$  be $OE$-groups having a finite common subgroup $H$. Suppose that  there is no isomorphism $\gamma:G_1\rightarrow G_2$ such that $\gamma(H)=H$. Then the number of isomorphism classes of amalgamated free products  $G_1*_{H,\mu_2}G_2$, $\mu_2\in Aut(H)$ is 
 
$$|\overline{Aut}_{G_2}(H)\backslash Aut(H)/ \overline{Aut}_{G_1}(H)|.$$
\end{cor}
\begin{proof}  Follows directly from Proposition \ref{gamma} (i).
\end{proof}
   
For a subset $S\subseteq Aut(H)$ we use $\widetilde S$ to denote its image in $Out(H)$. 

\begin{thm} \label{fixed images out} Let    $G_1$ and $G_2$  be $OE$-groups with isomorphic finite subgroups $H_1$ and $H_2$. Suppose that  there is no isomorphism $\gamma:G_1\rightarrow G_2$ such that $\gamma(H)=H$. Then the number of isomorphism classes of  amalgamated free products  $G_1*_{H,\mu_2}G_2$, $\mu_2\in Aut(H)$     is 
 
$$ |\widetilde{Aut}_{G_2}(H)\backslash Out(H)/ \widetilde{Aut}_{G_1}(H)|.$$
            
\end{thm}            
\begin{proof} Let $Inn_{G}(H)$ 
be the set of all inner automorphisms of a group $G$  that fixe a subgroup $H$ and denote by 
$\overline{Inn}_{G}(H)$ the image of $Inn_{G}(H)$ in $Aut(H)$. Then 
$Inn(H) \leq \overline{Inn}_{G_i}(H) \leq  \overline{Aut}_{G_i}(H)$. Therefore we can factor 
$Inn(H)$ out in the formula of Corollary \ref{fixed images isomorphisms number} to get the needed expression. 
\end{proof}

Let $G_1,G_2$ be $OE$-groups and $H_1=H=H_2$ be their common finite subgroup.  Suppose now  there is an isomorphism $\gamma:G_1\longrightarrow G_2$ with $\gamma(H)=H$.
 
Define an action of $C_2=<x>$ on $\overline{Aut}_{G_{2}}(H)\backslash Aut(H)/\overline{Aut}_{G_1}(H)$ as follows  
\begin{equation}\label{action}
x . \left( \overline{Aut}_{G_2}(H) \alpha \overline{Aut}_{G_1}(H)\right)=\overline{Aut}_{G_2}(H) \gamma\alpha^{-1}\gamma\overline{Aut}_{G_1}(H), 
\end{equation}
for $\alpha\in Aut(H)$.

Then we can deduce from Proposition \ref{gamma} (ii)  the following

\begin{cor}\label{symmetrical fixed images isomorphism number} Let    $G_1$ and $G_2$  be $OE$-groups having a finite common subgroup $H$. Suppose that  there is an isomorphism $\gamma:G_1\rightarrow G_2$ such that $\gamma(H)=H$.  Then the number of isomorphism classes of amalgamated free products $G_1*_{H,\mu_2}G_2$ with $\mu_2\in Aut(H)$  is $$|(\overline{Aut}_{G_2}(H)\backslash Aut(H)/\overline{Aut}_{G_1}(H))/C_2|,$$ where $C_2$ acts as defined in (\ref{action}).
\end{cor}


Now  put $\tilde\gamma=\gamma Inn(H)\in Out(H)$. We define an action of the $C_2=<x>$ on $\widetilde{Aut}_{G_2}(H)\backslash Out(H)/\widetilde{Aut}_{G_1}(H)$ as follows

\begin{equation}\label{action2}
x. \left( \widetilde{Aut}_{G_2}(H) \widetilde \alpha \widetilde{Aut}_{G_1}(H)\right)=\widetilde{Aut}_{G_2}(H) \tilde\gamma \widetilde \alpha^{-1}\tilde\gamma\widetilde{Aut}_{G_1}(H), 
\end{equation}
for $\widetilde \alpha\in Out(H)$.\\

As in previous case we can factor out $Inn(H)$ from the formulas in Corollary \ref{symmetrical fixed images isomorphism number} to get the following

\begin{cor}\label{symmetrical fixed images isomorphism number2} Let    $G_1$ and $G_2$  be $OE$-groups having a finite common subgroup $H$. Suppose that  there is an isomorphism $\gamma:G_1\rightarrow G_2$ such that $\gamma(H)=H$. Then the number of isomorphism classes of  amalgamated free products $G_1*_{H,\mu_2}G_2$  is
$$|(\widetilde{Aut}_{G_2}(H)\backslash Out(H)/\widetilde{Aut}_{G_1}(H))/C_2|,$$ where $C_2$ acts as defined in (\ref{action2}).

\end{cor}

\begin{rem} \label{finite-by-cyclic profinite isomorphism problems2} 
The results of Section \ref{section4} also hold if we assume that $G_i\,(i=1,2)$ are semidirect products $M_i \rtimes \mathbb{Z}$ with $M_i$ finite such that $M_i\neq H$.
This follows from Remark \ref{finite-by-cyclic} (i).
\end{rem}


\section{Profinite isomorphism problem} \label{Profinite isomorphism problem}


In Theorem \ref{6} we proved  that the amalgamated free  products $G_1*_{\lambda, H,\mu} G_2$  and $G_1*_{\eta, H, \nu} G_2$ are isomorphic if and only if their push-outs $(\lambda,\mu)$ and $(\eta,\nu)$ are isomorphic. For  profinite amalgamated free products $G=G_1\amalg_{\lambda, H,\mu} G_2$  and $B=G_1\amalg_{\eta, H, \nu} G_2$ it is true only up to conjugation of one of the factors by an element of the normalizer $N_{G}(H)$. As in the abstract case denote by $\overline{Aut}_{G_i}(H)$ the natural image of $Aut_{G_i}(H)$ in $Aut(H)$.

Let $G=G_1\coprod_{\lambda,H,\mu} G_2$ be an  amalgamated free profinite product
of $OE$-groups with $H$ finite. By abuse of notation denote $N_{ G}(\lambda(H))=N_{ G}(\mu(H))$ simply as $N_{ G}(H)$ (viewing $H$ as a subgroup of $G$). Then for any $r\in N_{ G}(H)$, $\tau_r$ leaves $H$ invariant. 


With these notations we have the following Theorem.
    

\begin{thm}\label{1}  
 Let $\lambda,\eta:H\longrightarrow G_1$ and $\mu, \nu: H\longrightarrow G_2$ be monomorphisms of a finite group $H$ to profinite $OE$-groups $G_1$, $G_2$. Let  $G=G_1\amalg_{\lambda, H,\mu} G_2$  and $B=G_1\amalg_{\eta, H, \nu} G_2$ be  amalgamated free profinite products. Then $ G\cong  B$ if and only if (up to possibly interchanging  $G_1, G_2$ in the case $G_1\cong G_2$) there exist $\beta_{1}\in Aut(G_1)$, $\beta_{2}\in Aut(G_2)$ and $r\in N_{G}(H)$ with $\overline{\langle G_1, G_2^r\rangle}=G$ such that $(\beta_1,\beta_2, \alpha):(\eta,\nu)\longrightarrow (\lambda,\mu\tau_r)$ is an isomorphism of push-outs for certain $\alpha\in Aut(H)$. 
\end{thm}

\begin{proof} $(\Leftarrow)$ Note that  $G=\overline{\langle G_1, G_2^r\rangle}= G_1\amalg_{\lambda, H,\mu}  G_2$ by  \cite[Proposition 4.4]{ZM} or  \cite[Theorem 6.6.1]{R}. Therefore 'if' follows.   

\medskip

$(\Rightarrow)$ By Proposition \ref{compart1}, there exists an isomorphism $\psi:B\longrightarrow G$  such  that $\psi(G_1)=G_1$, $\psi(G_2)=G_2^{g}$, for some $g\in G$, and $\psi(H)=H$. 

On the other hand, we have $H^{g^{-1}}=G_1^{g^{-1}}\cap G_2=H^{x}$ for some $x\in G_{2}$ (by Proposition \ref{o421}), whence  $H^{xg}=H$ and so $xg \in N_{G}(H).$ Denote $r:=xg$ and observe that $G_2^g=G_2^r$. To complete the proof define 
$$\beta_{1}=\psi\mid_{G_{1}} \in Aut(G_1)\,\mbox{ and }\, \beta_{2}= \tau_{r} \circ \psi\mid_{G_2} \in Aut(G_2).$$
Then $$(\tau_{r}^{-1}\beta_2\nu)(h)=\psi(\nu(h))=\psi(\eta(h))=(\beta_{1}\eta)(h)= \lambda(\lambda^{-1}\beta_{1}\eta)(h)=\mu(\lambda^{-1}\beta_{1}\eta)(h),$$ for all $h \in H.$ Therefore 
$(\tau_r\mu)^{-1}\beta_2\nu=\lambda^{-1}\beta_{1}\eta$ can be defined to be $\alpha\in Aut(H)$ (cf. Equation (\ref{pushout morphism})). 
The proof is complete because $(\beta_1,\beta_2, \alpha):(\eta,\nu)\longrightarrow (\lambda,\mu\tau_r)$ is an isomorphism of push-outs.
\end{proof}





  To estimate the number of profinite  free products with amalgamation up to isomorphism we note that as in the abstract case  $Aut(H)$ acts on $Inj(H,G_i)$ naturally on the right by $\varphi \cdot\alpha=\varphi\alpha$, $\varphi\in Inj(H,G_i)$, $\alpha\in Aut(H)$ and the automorphism group $Aut(G_i)$ acts naturally on $Inj(H,G_i)$ on the left $\beta \cdot \varphi=\beta\varphi$, $\beta\in Aut(G_i)$. 

\bigskip
Then  $Inj(H,G_1)\times Inj(H,G_2)$ gives the set  of profinite push-outs 

$$\xymatrix{G_1\ar[r]&G_1\amalg_H G_2\\
            H\ar[r]^{}\ar[u]^{} &G_2\ar[u]}$$ on which $Aut(G_1)\times Aut(G_2)$ acts on the left and we  consider the diagonal action of $Aut(H)$ on it on the right  by $(\lambda,\mu)\alpha=(\lambda\alpha, \mu\alpha)$, $\alpha\in Aut(H)$.  Thus we can interpret isomorphism of push-out diagrams in terms of this action. As in the abstract case we split our consideration into two cases: $G_1\not\cong G_2$ and $G_1\cong G_2$. Our next objective is to estimate the number of non-isomorphic   amalgamated free profinite products $G=G_1\amalg_H  G_2$ with $G_1,G_2$ $OE$-groups and  $H$ finite. 
            
            \subsection{$G_1\not\cong G_2$}
In this case the interpretation described in the latter paragraph can be stated as the following 
 
\begin{prop}\label{pushout profinite orbit}            
            Suppose  that  $G_1\not\cong G_2$. Then push-out diagrams 
            $$\xymatrix{G_1\ar[r]&G_1 \amalg_{\lambda, H, \mu} G_2\\
            H\ar[r]^{\mu}\ar[u]^{\lambda} &G_2\ar[u]} \hspace{3cm}
\xymatrix{G_1\ar[r]&G_1 \amalg_{\eta, H, \nu} G_2\\
            H\ar[r]^{\nu}\ar[u]^{\eta} &G_2\ar[u]}$$ are in the same  $(Aut(G_1)\times Aut(G_2), Aut(H))$-orbit  if and only if they are  isomorphic. In this case  $G_1\amalg_{\lambda, H, \mu}G_2\cong G_1\amalg_{\eta, H, \nu}G_2$. \end{prop}

\bigskip
Let 
$$\xymatrix{G_1\ar[r]&G=G_1\amalg_H G_2\\
            H\ar[r]^{}\ar[u]^{} &G_2\ar[u]}$$ be a profinite push-out. 
            
            To produce the formula for non-isomorphic  amalgamated free profinite products we should order this counting similar to the abstract situation above.

 Recall that   $Aut(G_i)$ acts on the left on $Inj(H,G_i)$. By Proposition \ref{pushout profinite orbit} the multiplication of $\mu$ by elements of $Aut(G_2)$ on the left or symmetrically the multiplication of $\lambda$ by   elements of $Aut(G_1)$ on the left  give isomorphic amalgamated free profinite products. Remember that
 we can change  $\mu$ and $\lambda$ also by multiplying them   by elements of $Aut(H)$ on the right   and multiplying $(\lambda,\mu)$ coordinatewise by $\alpha\in Aut(H)$ also gives isomorphism of the  amalgamated free (profinite) products (Proposition \ref{pushout orbit} (Proposition \ref{pushout profinite orbit})). In other words, $Aut(H)\times Aut(H)$ acts on $Inj(H,G_1)\times Inj(H,G_2)$ on the right and the action of the diagonal gives isomorphism of the  amalgamated free profinite products. It is not if and only if however (cf. Theorem \ref{1}). Indeed multiplying $\mu$ by elements of $\overline{N}_G(H)$ might not produce the morphism of the push-out diagrams but  still can produce isomorphic amalgamated  free profinite  products. Recall also that  $N_G(H)$ depends on the choice of a push-out ($\lambda,\mu)$.
 
  Therefore in the profinite case we can only give an estimate. 
 Namely combining Theorem \ref{1} and Proposition \ref{pushout profinite orbit} we deduce the following

\begin{thm}\label{profinite isomorphism number} Let $G_1\not\cong G_2$ be  profinite  $OE$-groups  and $H$  a finite  group. Then the number  of isomorphism classes of amalgamated free profinite products $G_1\amalg_H G_2$  is

   $$\leq |\left( \left( Aut(G_1)\backslash Inj(H,G_1) \right) \times \left( Aut(G_2)\backslash Inj(H,G_2) \right) \right)/ Aut(H)|,$$ where $Aut(H)$ acts diagonally.
   
\end{thm}

\subsection{$G_1\cong G_2$}
\bigskip
Suppose now $G_1\cong G_2$. As in the abstract case, we  think of $G_1,G_2$ as the same group $G$ and consider the set $Inj(H,G)$. Then as before $Aut(H)$ acts on the set of push-outs $(\lambda,\mu)$ diagonally and the wreath product $W=Aut(G) \wr C_2$ acts on the set of push-outs $(\lambda,\mu)$ on the left, with $C_2$ permuting $\lambda$ and $\mu$.
Hence  Theorem \ref{profinite isomorphism number} in this case reads as follows.

\begin{thm}\label{profinite symmetrical isomorphisms number} Let $H$ be a  finite group and $G$ be  a profinite  $OE$-group. Then  the number of the isomorphism classes of amalgamated free profinite products $G\amalg_H \widetilde{G}$ with $G\cong \widetilde{G}$  is  $$\leq |(Aut(G)\wr C_2)\backslash (Inj(H,G)\times Inj(H,G))/Aut(H)|,$$
 where $Aut(H)$ acts diagonally. 
            
            \end{thm}

\subsection{Amalgamating given subgroups}\label{profinite amalgamating given subgroups}

\medskip

In this section we fix profinite $OE$-groups $G_1,G_2$ with isomorphic finite  subgroups $H_1\leq G_1$, $H_2\leq G_2$. Consider the family $\mathcal{P}=\mathcal{P}(G_1,G_2,H)$ of all  free profinite products $B=G_1\amalg_{H_1=H_2} G_2$ with  $H_1$ and $H_2$ amalgamated.  We can interpret $H_1,H_2$ as  two fixed  copies of a finite group $H$ in $G_1$ and $G_2$ respectively. Then we can write $G=G_1 \amalg_{\lambda, H, \mu}G_2$, where  $\lambda \in Inj(H,H_1)$ and $\mu \in Inj(H,H_2)$. 

The profinite amalgamated free products $G$ can be seen as push-outs $(\lambda,\mu)$. However, from Theorems \ref{profinite isomorphism number} and  \ref{profinite symmetrical isomorphisms number} we saw that diagonal action preserves the isomorphism of push-outs, that is, we can reduce the problem of counting the classes of isomorphisms in $\mathcal{P}$, to counting the classes of isomorphisms of the push-outs of the form $( id, \mu)$. Thus  as in the abstract case slightly abusing notation, we assume w.l.o.g that $\lambda=id$. 

Note that $N_G(H)$ depends only on the fixed images of the push-out $(\lambda,\mu)$ in $G_1$ and $G_2$. Indeed,  in  Proposition \ref{pushout  profinite orbit}  the images $H_1=\lambda(H)$ and $H_2=\mu(H)$ of $H$ in $G_i$ are fixed in our case and so $Inj(H,G_1) \times Inj(H,G_2)$ are replaced by $Inj(H,H_1)\times Inj(H,H_2)\cong Aut(H)\times Aut(H)$ and $Aut(G_1), Aut(G_2)$ are replaced by $\overline{Aut}_{G_1}(H_1)$, $\overline{Aut}_{G_2}(H_2)$, respectively. But by Proposition \ref{profinite normalizer}    $N_G(H)$ is generated by $N_{G_1}(H_1)$ and $N_{G_2}(H_2)$. On the other hand, $\overline{N}_{G_i}(H_i)$ is normal in $\overline{Aut}_{G_i}(H_i)$, i.e. $\overline{N}_{G_1}(H_1)$ does not depend on multiplication of $\lambda$ by automorphisms of $\overline{Aut}_{G_1}(H_1)$ and $\overline{N}_{G_2}(H)$ does not depend on multiplication of $\mu$ by automorphisms of $\overline{Aut}_{G_2}(H_2)$. 

Define $N^{+}$ to be the following non-empty set $$N^+=\{n\in N_{ G}(H)\mid \overline{\langle  G_1,  G_2^n\rangle}= G\}.$$ Further denote by $\overline{ N}^+$ and $\widetilde N^+$ their natural images in $Aut(H)$ and $Out(H)$, respectively. Define $\overline{Aut}_{G_2}^{\mu}(H)$  similarly as in the abstract case. Then we have the following proposition.


\begin{prop}\label{genero} Let $G=G_1 \amalg_{H,\mu }G_2$ and $B=G_1 \amalg_{H,\mu \alpha}G_2$ for some $\alpha\in Aut(H)$.
\begin{enumerate} 

\item[i)]Suppose that  there is no isomorphism $\gamma: G_1\rightarrow  G_2$ such that $\gamma(H)=H_2=\mu(H)$. Then $ G\cong B $ if and only if $\alpha\in \overline{Aut}^{\mu}_{ G_2}(H)\overline{N}^+\overline{Aut}_{G_1}(H)$.
\item[ii)]Suppose there is an isomorphism $\gamma: G_1\rightarrow  G_2$ such that $\gamma(H)=H_2=\mu(H)$. Then $ G\cong  B$ if and only if  $\alpha\in \overline{Aut}^{\mu}_{G_2}(H)\overline{N}^+\overline{Aut}_{G_1}(H)$ or  
$\xi\alpha^{-1}\xi\in \overline{Aut}^{\mu}_{G_2}(H)\overline{N}^+\overline{Aut}_{ G_1}(H)$, where $\xi=\mu^{-1}\gamma\in Aut(H)$.
\end{enumerate}
\end{prop}
\begin{proof}
 By Proposition \ref{compart1} and Theorem \ref{1}  $ G\cong  B$ if and only if there exists an isomorphism $\psi: B\rightarrow  G$ such that $\psi( G_1)= G_1$ and $\psi( G_2)= G_2^r$ or  $\psi( G_1)= G_2^{r}$ and $\psi( G_2)= G_1$ for some $r\in N_{G}(H)$ such that $ G=\overline{\langle  G_1, G_2^r\rangle}$ (note that the second case only occurs when $ G_1\cong  G_2$ and $\psi|_{ G_1}: G_1\rightarrow  G_2$ is an isomorphism that satisfies $\psi(H)=H_2$). From now on the proof will be divided into two cases as follows.\\
  
Case i) If $\psi( G_1)=G_1$ and $\psi( G_2)= G_2^r$ we define $\beta_1=\psi|_{G_1} $ and $\beta_2=\tau_r\psi|_{ G_2}$. Then for all $h\in H$ we have $$\tau_r^{-1}\beta_2\mu\alpha(h)=\psi\mu\alpha(h)=\psi(h)=\beta_1(h)=\mu\beta_1(h)$$ or equivalently $\alpha=\beta_2^{-\mu}\tau_r^{\mu}\beta_1 \in \overline{Aut}^{\mu}_{ G_2}(H)\overline{N}^+\overline{Aut}_{G_1}(H)$.\\ 
 
Case ii) If $\psi( G_1)= G_2^r$ and $\psi( G_2)= G_1$ then by Remark \ref{trocaprofinite} there exist an isomorphism $\hat\theta:B \rightarrow  G_1 \amalg_{ H,  \gamma (\mu\alpha)^{-1} \gamma} G_2 $ that swaps $ G_1$ and $ G_2$, moreover $(\psi\hat\theta^{-1})( G_1)=G_1$ and $(\psi\hat\theta^{-1})( G_2)= G_2^r$. Observe that  $\gamma(\mu\alpha)^{-1}\gamma(h)= \mu(\mu^{-1}\gamma(\mu\alpha)^{-1}\gamma)(h)=\mu( \xi\alpha^{-1}\xi)(h)$ for all $h\in H$.  Define $\beta_1=(\psi\hat\theta^{-1})|_{G_1} $, $\beta_2=\tau_r(\psi\hat\theta^{-1})|_{ G_2}$. Then for all $h\in H$, we have $$\tau_r^{-1}\beta_2(\mu\xi\alpha^{-1}\xi)(h)=
\psi\hat\theta^{-1}(\mu\xi\alpha^{-1}\xi)(h)=\psi\hat\theta^{-1}(h)=\beta_1(h)=\mu\beta_1(h)$$
equivalently $\xi\alpha^{-1}\xi=(\beta_2^{-1})^{\mu}\tau_r^{\mu}\beta_1 \in \overline{Aut}^{\mu}_{ G_2}(H)\overline{N}^+\overline{Aut}_{G_1}(H)$.

\end{proof}

Our goal now is to find formulas to calculate the number of isomorphism classes of the amalgamated free products of profinite $OE$-groups $G_1$ and $G_2$ with finite amalgamation of fixed subgroups $H_1\leq G_1$, $H_2\leq G_2$. Thus interpreting $H_1,H_2$ as one common subgroup $H_1=H=H_2$ of $G=G_1 \coprod_{H} G_2$ (i.e assuming w.l.o.g. $\mu=id$) 
 Theorem \ref{profinite isomorphism number} admits the following simple form in this case.
 
 \begin{cor}\label{profinite fixed images isomorphisms number} Let    $G_1$ and $G_2$  be profinite $OE$-groups having a finite common subgroup $H$. Suppose that  there is no isomorphism $\gamma:G_1\rightarrow G_2$ such that $\gamma(H)=H$. Then the number of isomorphism classes of profinite free amalgamated  products  $G_1\amalg_{H,\mu_2}G_2$, $\mu_2\in Aut(H)$ does not exceed
 
$$|\overline{Aut}_{G_2}(H)\backslash Aut(H)/ \overline{Aut}_{G_1}(H)|.$$
\end{cor}
\begin{proof}  Follows directly from Proposition \ref{genero} (i).
\end{proof}

Recall that for a subset $S\subseteq Aut(H)$ we use $\widetilde S$ to denote its image in $Out(H)$. 

\begin{cor}\label{profinite fixed images out} Let    $G_1$ and $G_2$  be profinite $OE$-groups having a finite common subgroup $H$. Suppose that  there is no isomorphism $\gamma:G_1\rightarrow G_2$ such that $\gamma(H)=H$. Then the number of isomorphism classes of  profinite free amalgamated  products  $G_1\amalg_{H,\mu_2}G_2$, $\mu_2\in Aut(H)$     does not exceed 
 
$$ |\widetilde{Aut}_{G_2}(H)\backslash Out(H)/ \widetilde{Aut}_{G_1}(H)|.$$
            
\end{cor}            
\begin{proof} Let $Inn_{G}(H)$ 
be the set of all inner automorphisms of a group $G$  that fixes a subgroup $H$ and denote by 
$\overline{Inn}_{G}(H)$ the image of $Inn_{G}(H)$ in $Aut(H)$. Then 
$Inn(H) \leq \overline{Inn}_{G_i}(H) \leq  \overline{Aut}_{G_i}(H)$. Therefore we can factor 
$Inn(H)$ out in the formula of Corollary \ref{profinite fixed images isomorphisms number} to get the needed expression. 
\end{proof}

\bigskip
  Suppose now  there is an isomorphism $\gamma:G_1\longrightarrow G_2$ with $\gamma(H)=H$.
 As in the abstract case define an action of $C_2=\left\langle x \right\rangle$ on $$\overline{Aut}_{G_{2}}(H)\backslash Aut(H)/\overline{Aut}_{G_1}(H)$$ as follows  
\begin{equation}\label{profinite action}
x . \left( \overline{Aut}_{G_2}(H) \alpha \overline{Aut}_{G_1}(H)\right)=\overline{Aut}_{G_2}(H) \gamma\alpha^{-1}\gamma\overline{Aut}_{G_1}(H), 
\end{equation}
for $\alpha\in Aut(H)$.

Then we can deduce from Proposition \ref{genero} (ii)  the following

\begin{cor}\label{profinite symmetrical fixed images isomorphism number} Let    $G_1$ and $G_2$  be profinite $OE$-groups having a finite common subgroup $H$. Suppose that  there is an isomorphism $\gamma:G_1\rightarrow G_2$ such that $\gamma(H)=H$.  Then the number of isomorphism classes of profinite free amalgamated  products $G_1\amalg_{H,\mu_2}G_2$ with $\mu_2\in Aut(H)$  does not exceed $$|(\overline{Aut}_{G_2}(H)\backslash Aut(H)/\overline{Aut}_{G_1}(H))/C_2|,$$ where $C_2$ acts as defined in (\ref{profinite action}).
\end{cor}

\bigskip
Now  put $\tilde\gamma=\gamma Inn(H)\in Out(H)$. We define an action of  $C_2=\left\langle x \right\rangle$ on $\widetilde{Aut}_{G_2}(H)\backslash Out(H)/\widetilde{Aut}_{G_1}(H)$ as follows

\begin{equation}\label{profinite action2}
x. \left( \widetilde{Aut}_{G_2}(H) \widetilde \alpha \widetilde{Aut}_{G_1}(H)\right)=\widetilde{Aut}_{G_2}(H) \tilde\gamma \widetilde \alpha^{-1}\tilde\gamma\widetilde{Aut}_{G_1}(H), 
\end{equation}
for $\widetilde \alpha\in Out(H)$.\\

As in previous case we can factor out $Inn(H)$ from the formulas in Corollary \ref{profinite symmetrical fixed images isomorphism number} to get the following

\begin{cor}\label{symmetrical profinite fixed images isomorphism number2} Let    $G_1$ and $G_2$  be profinite $OE$-groups having a finite common subgroup $H$. Suppose that  there is an isomorphism $\gamma:G_1\rightarrow G_2$ such that $\gamma(H)=H$. Then the number of isomorphism classes of  amalgamated free products $G_1\amalg_{H,\mu_2}G_2$  does not exceed
$$|(\widetilde{Aut}_{G_2}(H)\backslash Out(H)/\widetilde{Aut}_{G_1}(H))/C_2|,$$ where $C_2$ acts as defined in (\ref{profinite action2}).

\end{cor}

\begin{rem} \label{finite-by-cyclic profinite isomorphism problems} 
The results of Section \ref{Profinite isomorphism problem}  also hold if we assume that $G_i\,(i=1,2)$ are semidirect products $M_i \rtimes \widehat{\mathbb{Z}}$ with $M_i$ finite such that $M_i\neq H$.
This follows from Remark \ref{finite-by-cyclic} (ii).
\end{rem}

We finish the section with a proposition that describes when $G_1\amalg_{H,\mu} G_2$ is isomorphic to a double. 

\begin{prop}\label{isomorphic to double} Let $G_1*_{H,\mu} G_2$ ( $G_1\amalg_{H,\mu} G_2$) be a (profinite)  amalgamated free product of (profinite) $OE$-groups with finite amalgamation.  Suppose
there is an isomorphism $\gamma:G_1 \longrightarrow G_2$ such that $\gamma_{|H}=\mu$. Then $G_1*_{H,\mu} G_2$ ($G_1\amalg_{H, \mu}G_2$) is isomorphic to a (profinite) double $G_1*_H G_1$ ($G_1\amalg_H G_1$).
\end{prop}
\begin{proof}
Observe that the isomorphisms $id: G_1\longrightarrow  G_1$ and  $\gamma^{-1}:  G_2\longrightarrow  G_1$ is a morphism of push-outs and hence induces   an isomorphism $\psi:G_1*_{H, \mu}G_2 \longrightarrow G_1*_H G_1$ ($\psi:G_1\amalg_{H, \mu}G_2 \longrightarrow G_1\amalg_H G_1$).
\end{proof}

\section{Genus when given subgroups are amalgamated} \label{secao6}

In this section we fix $\Oe$-groups $G_1,G_2$ and their finite common subgroup $H$. This means that there  are two fixed  copies $H_1, H_2$ of $H$ in $G_1$ and $G_2$ that we think of as a common subgroup $H$. We consider the family $\mathcal{W}=\mathcal{W}(G_1,G_2,H)$ of all amalgamated free products $B=G_1*_{H_1=H_2} G_2$.  The objective of the section is to find formulas and limitations for genus $g(G,\mathcal{W})$ with respect to this family. 

 Then the results of Section \ref{given subgroups}  and \ref{profinite amalgamating given subgroups} become relevant. In fact to simplify the notation we shall indetify $H$ with $H_1$ and $H_2$ and apply results of Section \ref{given subgroups} and w.l.o.g. \ref{profinite amalgamating given subgroups}  with $\lambda=id=\mu$. 
\begin{deff} We say that  $G=G_1*_{H}G_2$ is symmetric (resp. profinitely symmetric) if there is an isomorphism  $\gamma: G_1 \rightarrow G_2$ (resp. continuous isomorphism $\gamma: \widehat G_1 \rightarrow \widehat G_2$) such that $\gamma(H)=H$. 
\end{deff}

\begin{lem}\label{equality} Let $G=G_1*_{H}G_2$ (resp. $G=G_1\amalg_{H}G_2$) be a symmetric (resp. profinitely symmetric) free product of $OE$-groups with finite amalgamation and $\gamma: G_1 \rightarrow G_2$ be an isomorphism (resp. continuous isomorphism)  such that $\gamma(H)=H$. Then $\overline{Aut}_{ G_1}(H)=\left( \overline{Aut}_{ G_2}(H)\right)^{\xi}$ and $\widetilde{Aut}_{ G_1}(H)=\left( \widetilde{Aut}_{ G_2}(H)\right)^{\tilde\xi}$, where $\xi=\gamma|_{H}$ and $\tilde\xi$ is image of $\xi$ in $Out(H)$. 
\end{lem} 

\begin{proof} 
Take $\varphi$ in ${Aut}_{ G_2}(H).$ Then for $h\in H$ we have $\xi^{-1}\varphi\xi(h)=\gamma^{-1}\varphi\gamma(h)$. So $\gamma^{-1}\varphi\gamma\in \overline{Aut}_{ G_1}(H)$.  To complete the proof one  takes the images of the preceding equality in $Out(H).$
\end{proof}

The following result is an immediate consequence of Lemma \ref{equality} and Proposition \ref{cp1}. 

\begin{cor}\label{normalizer in Out(H)} 
Let $G=G_1*_{H}G_2$ be a symmetric free product of $\Oe$-groups with finite amalgamation and let $\gamma:G_1 \longrightarrow G_2$ be an  isomorphism such that $\gamma(H)=H$ and $\xi=\gamma|_{H}$. Then 
$$\overline N_G(H)=\langle \overline N_{G_2}(H), \overline N_{G_2}(H)^{\xi}\rangle, \overline N_{\widehat G}(H)=\langle \overline N_{\widehat G_2}(H), \overline N_{\widehat G_2}(H)^{\xi}\rangle$$ and 
$\widetilde N_G(H)=\langle \widetilde N_{G_2}(H), \widetilde N_{G_2}(H)^{\tilde\xi}\rangle$, $\widetilde N_{\widehat G}(H)=\langle \widetilde N_{\widehat G_2}(H), \widetilde N_{\widehat G_2}(H)^{\tilde\xi}\rangle$.
\end{cor}

\bigskip
Now we divide our consideration into three cases: profinitely non-symmetric, non-symmetric but profinitely symmetric and symmetric. 

\subsection{Profinitely non-symmetric case}

In this subsection we will consider the case where $G=G_1*_{H}G_2$ is not profinitely symmetric, i.e.  there is no isomorphism $\gamma:\widehat G_1\rightarrow \widehat G_2$ such that $\gamma(H)=H$. 
\\
 By Proposition \ref{genero} (i) combined with Corollary \ref{fixed images isomorphisms number} and Theorem \ref{fixed images out}
  we  just need  to restrict the natural map $$\kappa:Out(H)\longrightarrow \widetilde{Aut}_{G_2}(H)\backslash Out(H)/\widetilde{Aut}_{G_1}(H)$$ to $\widetilde{Aut}_{\widehat G_2}(H)\widetilde N^+\widetilde{Aut}_{\widehat G_1}(H)$ in this case, where  $\widetilde N^+$ is the natural image of the subset $N^+=\{n\in N_{\widehat G}(H)\mid \overline{\langle  G_1,  G_2^n\rangle}=\widehat G\}$ in $Out(H)$. In fact, we can think of $\K=\widetilde{Aut}_{\widehat G_2}(H)\widetilde N^+\widetilde{Aut}_{\widehat G_1}(H)$ as a subset of $Out(H)$ on which $\widetilde{Aut}_{G_2}(H)$ acts on the left  and $\widetilde{Aut}_{G_1}(H)$ acts on the right and so $\kappa(\K)$  can be written as $\widetilde{Aut}_{G_2}(H)\backslash \K/\widetilde{Aut}_{G_1}(H)$.
  Thus we obtain the following 
  
\begin{thm}\label{fixed subgroups} Let $G=G_1*_{H}G_2$ be a profinitely non-symmetric free product of $\Oe$-groups with finite amalgamation.  Then the genus 
  $g(G,\W)$ equals
$$|\widetilde{Aut}_{G_2}(H)\backslash \widetilde{Aut}_{\widehat G_2}(H)\widetilde N^+\widetilde{Aut}_{\widehat G_1}(H)/\widetilde{Aut}_{G_1}(H)|.\eqno{(*)}$$
\end{thm}

If $H$ is a semidirect factor or self-normalized in $\widehat G_1$ or $\widehat G_2$ then $\widetilde{N}^+$ disappears from the formula $(*)$ and so we have the following

\begin{cor}\label{semidirect} 
Let $G=G_1*_{H}G_2$ be a profinitely non-symmetric free product of $\Oe$- groups with finite amalgamation. Suppose $G_1=N\rtimes H$ or $N_{\widehat G_1}(H)=H$. 
Then  $$g(G,\W)= |\widetilde{Aut}_{G_2}(H)\backslash \widetilde{Aut}_{\widehat G_2}(H) \widetilde{Aut}_{\widehat G_1}(H)/\widetilde{Aut}_{G_1}(H)|. $$

\end{cor} 

\begin{proof} In either case $N_{\widehat G_1}(H)=C_{\widehat G_1}(H)H$ and so $\widetilde N_{\widehat G_1}(H)=1$.

\end{proof}

The set $N^+$ is quite mysterious, since it is not clear how to decide whether $\overline{\langle  G_1, G_2^n\rangle}$ generate $\widehat G$ (see \cite[Remark 4.14]{GZ} in the case of finite groups).
If $\widetilde{N}_{\widehat G}(H)=\widetilde{N}_{\widehat G_2}(H)\widetilde N_{\widehat G_1}(H)$  then $\widetilde N^+$ disappears from the formula $(*)$. 

\begin{cor}\label{semidirect} Within hypothesis of Theorem \ref{fixed subgroups} suppose $$\widetilde N_{\widehat G}(H)=\widetilde N_{\widehat G_2}(H)\widetilde N_{\widehat G_1}(H).$$ 
Then  $$g(G,\W)= |\widetilde{Aut}_{G_2}(H)\backslash \widetilde{Aut}_{\widehat G_2}(H) \widetilde{Aut}_{\widehat G_1}(H)/\widetilde{Aut}_{G_1}(H)|. $$

\end{cor} 

\begin{proof}  Recall that $N^+\subseteq N_{\widehat G}(H)$ and $N_{\widehat G}(H)=\overline{\langle N_{\widehat G_1}(H),N_{\widehat G_2}(H)\rangle}$ by Proposition \ref{cp1} (ii). So $\widetilde N^+\subseteq \widetilde N_{\widehat G_2}(H)  \widetilde N_{\widehat G_1}(H)  \leq \widetilde{Aut}_{\widehat G_2}(H) \widetilde{Aut}_{\widehat G_1}(H)$.
\end{proof}

In the next theorem we list the situations when the condition of Corollary \ref{semidirect} is satisfied and so we can give an exact formula for the genus.

\begin{thm}\label{3 formula} Let $G=G_1*_H G_2$ be a profinitely non-symmetric amalgamated free product of $\Oe$-groups $G_1,G_2$ with finite common subgroup $H$. 
 Then  $g(G,\mathcal{W})=|\widetilde{Aut}_{G_2}(H)\backslash \widetilde{Aut}_{\widehat G_2}(H) \widetilde{Aut}_{\widehat G_1}(H)/\widetilde{Aut}_{G_1}(H)|$ if one of the following conditions is satisfied:
\begin{itemize}
\item[(i)] $H$ is central in $G_1$ or  in $G_2,$
\item[(ii)] $H$ is a direct factor of $N_{\widehat G_1}(H)$ or of $N_{\widehat G_2}(H)$,
\item[(iii)]  $Out(H)$ is abelian (in particular if $H$ is cyclic, also the case for almost all finite simple groups),
\item[(iv)] $H$ is self-normalizing in $\widehat G_1$ or $\widehat G_2$,
\item[(v)] $H$ is a retract of $G_1$ or $G_2$.

\end{itemize}

\end{thm}

\begin{proof}
\smallskip
 If (i) holds then either  $\widetilde{N}_{\widehat G}(H)=\widetilde{N}_{\widehat G_2}(H)$ or $\widetilde{N}_{\widehat G}(H)=\widetilde{N}_{\widehat G_1}(H)$ and so the result follows from Corollary \ref{semidirect}.  

\smallskip
If (ii), (iv) or (v) holds then  either $\widetilde{N}_{\widehat G_1}\left(H\right) = 1$ or $\widetilde{N}_{\widehat G_2}\left(H\right) = 1$ so by Corollary \ref{semidirect}  we also have the result.

\smallskip
 If (iii) holds then  $\widetilde{N}_{\widehat G_2}\left(H\right) \triangleleft \widetilde{N}_{\widehat G}\left(H\right)$ and  we have the equality $\widetilde N_{\widehat G}(H)= \widetilde{N}_{\widehat G_2}(H)\widetilde{N}_{\widehat G_1}(H)$. So again     Corollary \ref{semidirect} imply the result.
\end{proof}
The following corollary is a generalization  of \cite[Theorem 4.1]{GZ}, where $G_i$ were assumed finite.

\begin{cor}
With the hypotheses of Theorem \ref{3 formula} if $\widetilde{Aut}_{\widehat{G_i}}(H)=\widetilde{Aut}_{G_i}(H)$ for each $i=1,2$ then $g(G,\mathcal{W})=1$.
\end{cor}

\begin{rem}\label{denso non symmetric profinitely} Assume that  $\widetilde{Aut}_{\widehat G_i}(H)= \widetilde{Aut}_{G_i}(H)$.  
   Since $\widetilde{Aut}_{G_i}(H)\cap \widetilde N_{\widehat G}(H)=\widetilde N_{G_i}(H)$, one has $$\widetilde N_{\widehat G}(H)\widetilde{Aut}_{\widehat G_i}(H)/ 
 \widetilde{Aut}_{G_i}(H) = \widetilde N_{\widehat G}(H)\widetilde{Aut}_{ G_i}(H)/ 
 \widetilde{Aut}_{G_i}(H)= \widetilde N_{\widehat G}(H)/\widetilde N_{G_i}(H).$$  Then
   $(*)$   reduces to   the image of $N^+$ in
 $$\widetilde N_{ G_2}(H)\backslash \widetilde N_{\widehat G}(H)/\widetilde N_{G_1}(H).$$ In particular the genus $g(G, \W)$ in Theorem \ref{fixed subgroups} will be bounded by $$|\widetilde N_{ G_2}(H)\backslash \widetilde N_{\widehat G}(H)/\widetilde N_{G_1}(H)|.$$\end{rem}

\bigskip
If $G_i (i=1,2)$ are finite, then $\widetilde N_{\widehat G}(H)=\widetilde N_{G}(H)$  by Proposition \ref{cp1}. So we can deduce from  Remark \ref{denso non symmetric profinitely} a sharper estimate in this case.

\begin{cor}\label{dense estimate} If $G_i$ is finite for $i=1,2$ and $G=G_1*_{H}G_2$ is not symme\-tric (equivalently not profinitely symmetric)  then $$g(G, \W) \leq  |\widetilde N_{ G_2}(H)\backslash \widetilde N_{G}(H)/\widetilde N_{G_1}(H)|.$$
\end{cor}

\begin{rem} \label{case polycyclic}
Let $G_i$ $(i=1,2)$ be polycyclic-by-finite groups such that $\widetilde{Aut}_{\widehat G_i}(H)= \widetilde{Aut}_{G_i}(H)$ for each $i=1,2$. Let $G=G_1*_H G_2$ be a free product with finite amalgamation. Then $$g(G, \W) \leq  |\widetilde N_{ G_2}(H)\backslash \widetilde N_{G}(H)/\widetilde N_{G_1}(H)|$$ by Corollary \ref{8} and Remarks \ref{denso non symmetric profinitely}. 
\end{rem}



\subsection{Non-symmetric but profinitely symmetric case} \label{secao6.2}

 In this subsection we will consider the case where $G=G_1*_{H}G_2$ is non-symmetric but profinitely symmetric. So we  have an isomorphism $\gamma:\widehat G_1\rightarrow \widehat G_2$ such that $\gamma(H)=H$ but there is no isomorphism 
$\phi: G_1\rightarrow G_2$ such that $\phi(H)=H$. Put $\xi=  \gamma|_{H}\in Aut(H)$.


By Theorem \ref{fixed images out} we have the following natural  map 
$$\kappa: Out(H)\rightarrow \widetilde{Aut}_{G_2}(H)\backslash Out(H)/\widetilde{Aut}_{G_1}(H).$$ Let $\mathcal{K}=\widetilde{Aut}_{\widehat G_2}(H)\widetilde N^+\widetilde{Aut}_{\widehat G_1}(H)$ and $\tilde{\xi}$ be the image of $\xi \in Aut(H)$  in $Out(H)$.


For a subset $R$ of a group we shall use the notation $R^{-1}=\{r^{-1}\mid r\in R\}$.

\begin{lem} \label{equality S} With the notations and hypotheses of this subsection one has $$\\ \tilde\xi \mathcal{K}^{-1}\tilde\xi= \widetilde{Aut}_{\widehat G_2}(H)\tilde{\xi}\  (\widetilde{N}^+)^{-1}\ \tilde{\xi}\widetilde{Aut}_{\widehat G_1}(H).$$ 
\end{lem}

\begin{proof}
 For $\tilde{f_i} \in \widetilde{Aut}_{\widehat G_i}(H)$, $i=1,2$ and  $\tilde{w} \in \widetilde N^+$  we have 
 $$\tilde\xi(\tilde{f_2} \tilde{w} \tilde{f_1})^{-1} \tilde\xi = \tilde\xi \tilde{f_1}^{-1} \tilde{w}^{-1} \tilde{f_2}^{-1}  \tilde\xi=$$ 
 $$= (\tilde\xi \tilde{f_1}^{-1} \tilde{\xi}^{-1})(\tilde{\xi}\tilde{w}^{-1}\tilde{\xi})(\tilde{\xi}^{-1} \tilde{f_2}^{-1}  \tilde\xi)\in \widetilde{Aut}_{\widehat G_2}(H)\tilde{\xi} (\widetilde{N}^+)^{-1}\tilde{\xi} \widetilde{Aut}_{\widehat G_1}(H)$$ 
by Lemma \ref{equality}   as needed.
\end{proof}


It follows from  Proposition \ref{genero} (ii), Lemma \ref{equality S}, Corollary \ref{fixed images isomorphisms number} and Theorem \ref{fixed images out} that in this case we need to restrict the map $$\kappa: Out(H)\rightarrow \widetilde{Aut}_{G_2}(H)\backslash Out(H)/\widetilde{Aut}_{G_1}(H)$$ to $$\mathcal{S}=\mathcal{K}\cup \tilde\xi \mathcal{K}^{-1}\tilde\xi.$$ 
As in the previous subsection we can think of $\mathcal{S}$ as a subset of $Out(H)$ on which $\widetilde{Aut}_{G_2}(H)$ acts on the left  and $\widetilde{Aut}_{G_1}(H)$ acts on the right and so $\kappa(\mathcal{S})$  can be written as $\widetilde{Aut}_{G_2}(H)\backslash \mathcal{S}/\widetilde{Aut}_{G_1}(H)$.

Hence in this case the formula for genus is as follows.

\begin{thm}\label{fixed subgroups prof non sym} Let $G=G_1*_{ H} G_2$ be an  amalgamated free product of $\Oe$-groups $G_1,G_2$ with common finite subgroup $H$. Suppose  $G$ is not symmetric but profinitely symmetric. Then the  genus $g(G,\W)$ equals
$$|\widetilde{Aut}_{G_2}(H)\backslash \mathcal{S}/\widetilde{Aut}_{G_1}(H)|. \eqno{(**)}$$
\end{thm}

As in the preceding subsection if  $\widetilde{N}_{\widehat G}(H)=\widetilde{N}_{\widehat G_2}(H)\widetilde N_{\widehat G_1}(H)$ ( in which case  also $\widetilde{N}_{\widehat G}(H)=\widetilde{N}_{\widehat G_1}(H)\widetilde N_{\widehat G_2}(H)$ (see 1.3.13 in \cite{RO})) then $\widetilde{N}^{+}=\widetilde{N}_{\widehat G}(H)$, moreover,   $\widetilde{N}^{+}$ disappears from the formula for $\mathcal{K}$ and $\mathcal{S}$, i.e. in this case $\mathcal{K}=\widetilde{Aut}_{\widehat G_2}(H)\widetilde{Aut}_{\widehat G_1}(H)$ and
$\mathcal{S}=\widetilde{Aut}_{\widehat G_2}(H)\widetilde{Aut}_{\widehat G_1}(H)\cup \widetilde{Aut}_{\widehat G_2}(H)(\tilde{\xi})^{2}\widetilde{Aut}_{\widehat G_1}(H)$. Indeed,  
$$\tilde\xi(\widetilde{N}^{+})^{-1} \tilde\xi=\tilde\xi \widetilde{N}_{\widehat G}(H)\tilde\xi = \tilde\xi \widetilde{N}_{\widehat G_1}(H)\widetilde N_{\widehat G_2}(H) \tilde\xi=$$
 $$= (\tilde\xi \widetilde N_{\widehat G_1}(H) \tilde{\xi}^{-1})(\tilde{\xi}\tilde{\xi})(\tilde{\xi}^{-1} \widetilde N_{\widehat G_2}(H)  \tilde\xi)=\widetilde{N}_{\widehat G_2}(H) (\tilde{\xi})^{2} \widetilde N_{\widehat G_1}(H),$$ therefore
 \begin{equation} \label{Sprofsimmetric}
 \mathcal{S} =  \widetilde{Aut}_{\widehat G_2}(H) \widetilde{Aut}_{\widehat G_1}(H) \cup  \widetilde{Aut}_{\widehat G_2}(H)(\tilde{\xi})^{2} \widetilde{Aut}_{\widehat G_1}(H).
\end{equation}

We shall  use (\ref{Sprofsimmetric}) in the next subsections.

\subsubsection{Profinite double}\label{profinite double}

In this section we consider the case when $G=G_1*_HG_2$ is not symmetric but become a profinite double after profinite  completion  (see Proposition \ref{isomorphic to double} describing when it happens). In this case we can think of $\widehat G_2$ as isomorphic copy of $\widehat G_1$ with $G_1,G_2$ being dense subgroups of $\widehat G_1$;  so $\widehat G \cong D=\widehat G_1\amalg_H \widehat G_1$. 
 
Note that in this case we have $\widetilde{N}_{\widehat G}(H)=\widetilde{N}_{\widehat G_1}(H)$ by Proposition \ref{cp1} and moreover $\xi=id$. Then it follows from Equation (\ref{Sprofsimmetric}) and Lemma \ref{equality} that $\widetilde{N}^{+}$ disappears from the formula $\mathcal{K}$ and $\mathcal{S}$, i.e. in this case $$\mathcal{K}=\widetilde{Aut}_{\widehat G_1}(H)=\mathcal{S}.$$ Lemma \ref{equality} and Theorem \ref{fixed subgroups prof non sym}  give us the  following



\begin{thm}\label{genus 1 non symm but prof} Let $G=G_1*_{ H, \mu} G_2$ be a non-symmetric   amalgamated free product of $\Oe$-groups $G_1,G_2$ with common finite   subgroup $H$. Suppose  $\widehat G$ is isomorphic to a profinite double $D=\widehat G_1\amalg_H \widehat G_1$.  Then $$g(G,\W)= |\widetilde{Aut}_{G_2}(H)\backslash 
\widetilde{Aut}_{\widehat G_1}(H)/\widetilde{Aut}_{G_1}(H)|.$$ 
\end{thm}


\begin{cor} \label{usarnoexemplo}
With the  hypotheses of Theorem \ref{genus 1 non symm but prof} if $\widetilde{Aut}_{\widehat{G_1}}(H)=\widetilde{Aut}_{G_1}(H)$  then $g(G,\mathcal{W})=1$.
\end{cor}





We finish this subsection with an example of amalgamated free 
 product of virtually cyclic groups which is relevant by Remark \ref{finite-by-cyclic genus fixed amalgam}.

\begin{exa} Let $C$ be a cyclic group of odd order $m > 3$ and  $\psi:\Z\longrightarrow Aut(C)$ be a homomorphism. Let $ C\rtimes_\psi \Z$ be the corresponding semidirect product and consider $G_1=(C\rtimes_\psi \Z)\times C_2$. Let $n\in \Z$ such that $\psi(n) \neq \pm \psi(1)$ is 
 generator of $Im(\psi)$. Define the homomorphism $\varphi:\Z\longrightarrow Aut(C)$ by setting $\varphi(1)=\psi(n)$. Put $G_2=(C\rtimes_\varphi \Z)\times C_2$. Let $G=G_1*_C G_2$. Since $C$ is characteristic in $G_1$ and $G_2$ (also  in $\widehat G_1$ and $\widehat G_2$) we have $Aut_{G_i}(C)=Aut(G_i)$ and $Aut_{\widehat G_i}(C)=Aut(\widehat G_i)$. Moreover $Aut(C)=Out(C)=\mathbb{Z}_{m}^{*}$. Note that $G$ is not symmetric because $\psi(n)\neq \pm \psi(1)$ (see Corollary 2.7 in \cite{GZ}) and is profinitely symmetric by \cite[Proposition 2.8]{GZ}.  
We also claim that $\overline{Aut}_{G_i}(C)=\widetilde{Aut}_{G_i}(C)=Aut(C)$ (consequently 
$\overline{Aut}_{\widehat G_i}(C)=\widetilde{Aut}_{\widehat{G}_i}(C)=Aut(C)$). Indeed, let $\alpha$ be an automorphism of $C$. The map $\theta: G_i \longrightarrow G_i (i=1,2)$ defined by $\theta|_{C_2}=id$, $\theta|_{\Z}=id$ and $\theta|_{C}=\alpha$ is an automorphism of $G_i$ that lifts $\alpha$. Hence by Proposition \ref{isomorphic to double} $\widehat G$ is isomorphic to a profinite double of $\widehat G_1$ over $C$. Thus it follows from Corollary \ref{usarnoexemplo} that $$g(G, \W)=| Out(C) \backslash Out(C)/Out(C)|=1.$$ 
\end{exa}

\subsection{Symmetric case}

In this subsection we will consider the case where $G=G_1*_{H_1=H_2}G_2$ is symmetric (and so is profinitely symmetric).  Identifying $H_1$ and $H_2$ with $H$ we have an isomorphism $\gamma: G_1\rightarrow G_2$ such that $\gamma(H)=H$ and recall the notation $\xi=\gamma|_{H}$.



By Proposition \ref{genero} (ii) combined with Corollary \ref{symmetrical fixed images isomorphism number} and Corollary \ref{symmetrical fixed images isomorphism number2} we  just need in this case to restrict the natural map
$$\kappa: Out(H)\rightarrow \left(\widetilde{Aut}_{G_2}(H)\backslash Out(H)/\widetilde{Aut}_{G_1}(H)\right) /C_2$$ 
to $\mathcal{S}$ defined in the preceding subsection. Recall that the action of    $C_2=\left  \langle x \right\rangle$ on $\widetilde{Aut}_{G_2}(H)\backslash Out(H)/\widetilde{Aut}_{G_1}(H)$ defined in equation  (\ref{action2}) is

$$
x\cdot\left( \widetilde{Aut}_{G_2}(H) \widetilde \alpha \widetilde{Aut}_{G_1}(H)\right)=\widetilde{Aut}_{G_2}(H) \widetilde \xi \widetilde \alpha^{-1}\widetilde\xi \widetilde{Aut}_{G_1}(H), 
$$
for $\widetilde \alpha\in Out(H)$, where $\widetilde{\xi}$ is the image of the $\xi$ in $Out(H)$.
This action
leaves $\mathcal{S}$ invariant.  Indeed, clearly $x\cdot\K=\widetilde \xi \mathcal{K}^{-1} \widetilde \xi$. On the other hand,  let   $\widetilde{Aut}_{\widehat G_2}(H) \widetilde{\xi} (\widetilde{\zeta})^{-1} \widetilde \xi \widetilde{Aut}_{\widehat G_1}(H)   \in  \widetilde \xi \mathcal{K}^{-1} \widetilde \xi$ (see Lemma \ref{equality S}). We have \begin{equation*}
x. \left( \widetilde{Aut}_{\widehat G_2}(H)  \widetilde \xi (\widetilde \zeta)^{-1}\widetilde\xi   \widetilde{Aut}_{\widehat G_1}(H)\right)= \widetilde{Aut}_{\widehat G_2}(H) \widetilde \xi \cdot ( (\widetilde \xi)^{-1} \widetilde \zeta (\widetilde \xi)^{-1}    )  \cdot   \widetilde\xi \widetilde{Aut}_{\widehat G_1}(H)=
\end{equation*}

\begin{equation*}
=\widetilde{Aut}_{\widehat G_2}(H)  \widetilde{\zeta} \widetilde{Aut}_{\widehat G_1}(H)   \in  \mathcal{K} \subset \mathcal{S}=\mathcal{K}  \cup \widetilde \xi \mathcal{K}^{-1}\widetilde\xi.
\end{equation*}

  Therefore $\kappa(\mathcal{S})$ can be written as $(\widetilde{Aut}_{G_2}(H)\backslash \mathcal{S}/\widetilde{Aut}_{G_1}(H))/C_2$. Hence in this case the formula for genus is as follows

\begin{thm}\label{case symmetrical genus} Let $G=G_1*_{H, \mu} G_2$ be a symmetric free product of $\Oe$-groups with finite amalgamation. Then the  genus $g(G,\W)$ equals
$$\left|\left(\widetilde{Aut}_{G_2}(H)\backslash \mathcal{S}/\widetilde{Aut}_{G_1}(H)\right)/C_2\right|. \eqno{(***)}$$
\end{thm}

\begin{rem}
As in the preceding subsection if  $\widetilde{N}_{\widehat G}(H)=\widetilde{N}_{\widehat G_2}(H)\widetilde N_{\widehat G_1}(H)$ ( in which case  also $\widetilde{N}_{\widehat G}(H)=\widetilde{N}_{\widehat G_1}(H)\widetilde N_{\widehat G_2}(H)$) then 
$\widetilde{N}^{+}$ disappears from the formula of $\mathcal{K}$ and $\mathcal{S}$, i.e. in this case $\mathcal{K}=\widetilde{Aut}_{\widehat G_2}(H)\widetilde{Aut}_{\widehat G_1}(H)$ and $\mathcal{S}=\widetilde{Aut}_{\widehat G_2}(H)\widetilde{Aut}_{\widehat G_1}(H)\cup \widetilde{Aut}_{\widehat G_2}(H)(\tilde\xi)^2\widetilde{Aut}_{\widehat G_1}(H)$ (see formula (\ref{Sprofsimmetric})). \end{rem}

If $\widehat G$ is isomorphic to a profinite double $D=\widehat G_1\amalg_H\widehat G_1$ (cf. Subsection \ref{profinite double} and Proposition \ref{isomorphic to double}), the  formula for the genus $g(G, \W)$ can be simplified.


\begin{thm}\label{genus 1 symm symm} Let $G=G_1*_H G_2$ 
be a symmetric amalgamated free product of $\Oe$-groups $G_1,G_2$ with finite common subgroup $H$. Suppose that  $\widehat G$ is isomorphic to a profinite double $D=\widehat G_1\amalg_H\widehat G_1$. Then $$g(G,\W)= \left|\left(\widetilde{Aut}_{G_2}(H)\backslash  \widetilde{Aut}_{\widehat G_1}(H)/\widetilde{Aut}_{G_1}(H)\right)/C_2\right|,$$ where $C_2$ acts by inversion. 
\end{thm}

\begin{proof} It is similar to
 the proof of Theorem \ref{genus 1 non symm but prof}. 
\end{proof}

\begin{cor}
With the  hypotheses of Theorem \ref{genus 1 symm symm} if $\widetilde{Aut}_{\widehat{G_1}}(H)=\widetilde{Aut}_{G_1}(H)$, then $g(G,\mathcal{W})=1$.
\end{cor} 


\subsubsection{Double}

Let us now consider an important particular case when $G=G_1 *_{H} G_2$ is isomorphic to a double of $G_1$ (Proposition \ref{isomorphic to double} gives conditions when this is the case). Then we can view $G_1,G_2$ as one group $G_1$ and write $G \cong D=G_1 *_{H} G_1$.  Recall that by  Corollary \ref{symmetrical fixed images isomorphism number2}  the number of isomorphism classes of amalgamated free products is 
  
  $$\left| \left( \widetilde{Aut}_{G_1}(H)\backslash Out(H)/\widetilde{Aut}_{G_2}(H)\right)/C_2 \right|$$ in the symmetric case.   
For a double the formula $(***)$  simplifies, since $\widetilde N_{\widehat D}(H)=\widetilde N_{\widehat G_1}(H)$ in this case (see Corollary \ref{normalizer in Out(H)}), $\tilde{\xi}=id$ and $$\widetilde{Aut}_{\widehat G_1} (H)=\widetilde{Aut}_{\widehat G_2} (H), \widetilde{Aut}_{G_1} (H)=\widetilde{Aut}_{G_2} (H).$$ Thus the natural map  $$\kappa: Out(H)\longrightarrow \left( \widetilde{Aut}_{G_1}(H)\backslash Out(H)/\widetilde{Aut}_{G_1}(H) \right) /C_2$$  has to  be restricted to $\widetilde{Aut}_{\widehat G_1}(H)$ in the case of a double. Thus for a double we have the following formula for the genus.

\begin{thm}\label{genus  double} Let $G_1$ be an $\Oe$-group and $H$ be a finite subgroup of $G_1$.  
Let $D=G_1*_{H}G_1$ be the double. Then $$g(D,\W)=|(\widetilde{Aut}_{G_1} (H)\backslash  \widetilde{Aut}_{\widehat G_1} (H) / \widetilde{Aut}_{G_1} (H))/C_2|,$$
where $C_2$ acts by inversion.
\end{thm}

\begin{cor} \label{genero 1 finito} If $\widetilde{Aut}_{G_1}(H)=\widetilde{Aut}_{\widehat G_1}(H)$ then $g(D, \W)=1.$ In particular, this is the case if $G_1$ is finite. 

\end{cor}

\begin{exa} Let $G_1= \langle r_1, c_1 | r_1^4=c_1^2=1, c_1r_1c_1^{-1}=r_1^{-1}\rangle \cong D_8$, $G_2=\langle r_2, c_2 | r_2^4=c_2^2=1, c_2r_2c_2^{-1}=r_2^{-1}\rangle \cong D_8$ and $H=\langle c_1, r_1^2\rangle \cong C_2 \times C_2$ be one of the Klein subgroups of $G_1$. Observe that any $G=G_1*_H G_2$ is a symmetric amalgamated free product. Let $\alpha$ be an automorphism of $H$ that applies $c_1$ to $c_1$ and $r_1^2$ to $c_1r_1^2.$ The double $D=G_1*_{H, id} G_2=G_1*_H G_1$ and $B=G_1*_{ H, \alpha} G_2$ are not isomorphic by Theorem \ref{6}. Note that $Aut(H)=Out(H)\cong S_3$. By Corollary \ref{symmetrical fixed images isomorphism number2} we have that the number of the isomorphism classes of amalgamated free products $G_1*_H G_2$ is equal to 2 since the two distinct orbits are given by the orbits  of the representatives  $id$ and $\alpha.$ However by Theorem \ref{1} we have that $\widehat D \not \cong \widehat B$. So the genus $g(G, \W)=g(D, \W)=g(B,\W)=1$.
\end{exa}


\begin{exa}
In the previous example, let $H=<r>$ be the only subgroup of order 4 of $G_1 \cong D_8 \cong G_2.$ In this case, any amalgamated free product $G=G_1 *_H G_2 $ is isomorphic to  a double by Proposition \ref{isomorphic to double} since $\overline{Aut}_{G_i}(H)=Aut(H) \cong C_2\,(i=1,2)$, which implies that $g(G, \W)=g(D, \W)=1$ by Corollary \ref{genero 1 finito}. If $H \cong  C_2$ the result is immediate since $Aut(H)$ is the trivial group. Therefore, for the amalgamated free products $G=D_8*_H D_8$ the genus $g(G, \W)$ is always equal to 1 regardless of the chosen amalgamated subgroup. 
\end{exa}

\begin{rem} \label{finite-by-cyclic genus fixed amalgam} It follows from Remarks \ref{finite-by-cyclic isomorphism problems}, \ref{finite-by-cyclic profinite isomorphism problems2} and \ref{finite-by-cyclic profinite isomorphism problems} that the results of Section \ref{secao6}  also hold if we assume that  $G_i$ ($i=1,2)$ are semidirect product $M_i\rtimes \Z$ such that $M_i\neq H$ and $|M_i|< \infty$ (respectively,  $M_i\rtimes \widehat{\Z}$ such that $M_i\neq H$ and $|M_i|< \infty$, in the profinite case).
\end{rem}

\section{Genus of a push-out}

Recall that $Aut(G_i)\,(i=1,2)$ acts naturally on the set $S_i$ of finite subgroups of $G_i$ and similarly $Aut(\widehat G_i)\,(i=1,2)$ acts naturally on the set $\widehat S_i$ of finite subgroups of $\widehat G_i$. Let $H$ be a common finite subgroup of $G_i\,(i=1,2).$ Denote by $Aut(G_i)H$ and $Aut(\widehat{G_i})H$ the orbits of $H$  with respect to these actions. We denote by $k_i=k_{G_i,H}$  the number of $Aut(G_i)$-orbits of $S_i\cap Aut(\widehat G_i)H$. Then   $Aut_{G_i}(H)$ and $Aut_{\widehat G_i}(H)$ are the stabilizers of  $H$ with respect to these actions. 

 Denote by $\F$ the family of all amalgamated free products $G_1*_H G_2$ of  $\Oe$-groups such that the embeddings of $H$ in $G_i \,(i=1,2)$ are not fixed. Note that $\W \subset \F$. Therefore to obtain the formula for the  genus of an amalgamated free product viewed as a push-out one just needs to sum the genera calculated in the previous section over $Aut(G_i)$-orbits of $S_i\cap Aut(\widehat G_i)H$.  Now  we can state the  theorems which gives the precise cardinality for the genus $\mathfrak g(G,\F)$, where $G=G_1*_H \,G_2$. 
This is a generalization of Theorem 4.10 \cite{GZ}. If one of the indices in the following summation is infinite or if any of the parcels is infinite, then we say that the genus is infinite.

\begin{thm}\label{genus amalgam} Let $G_1, G_2$ be  $\Oe$-groups and $G=G_1 *_H G_2$ be a  free product with finite amalgamation. Then  the genus  $g(G,\mathcal{F})$ of $G$  equals to  $g(G,\F)=\displaystyle\sum_{\{H_1,H_2\}} g(G_1*_{H_1=H_2}G_2, \W)$, where $\{H_1,H_2\}$ ranges over unordered pairs of representatives of $Aut(G_i)$-orbits of $S_i\cap Aut(\widehat G_i)H\,(i=1,2)$.

\end{thm}

Note that we do not need to distinguish symmetric and non-symmetric cases here, since the profinite completion of non-symmetric free product with amalgamation can be symmetric. Of course, the formula for $$g(G_1*_{H_1=H_2}G_2, \,\W)$$ under the summation takes the symmetry into account.



\begin{rem} \label{nilpotent} 
 If   $Aut(\widehat G_i)$-orbit of $H$ is finite,
then  we have that the number of $Aut(G_i)$-orbits of $Aut(\widehat G_i)H$ is
$|Aut(G_i)\backslash Aut(\widehat G_i)/Aut_{\widehat G_i}(H)|$. Denote by $k_i=k_{G_i, H}$ the number of $Aut(G_i)$-orbits of $Aut(\widehat G_i)H\cap S_i$. So we have an estimate for this number given by 
$$k_i=k_{G_i,H}\leq |Aut(G_i)\backslash Aut(\widehat G_i)/Aut_{\widehat G_i}(H)|.$$
Furthermore, if $S_i=\widehat S_i$ then the inequality becomes an equality. This is true for example if $G_i$ is finitely generated nilpotent because $tor(G_i)$ is finite  and $tor(G_i)=tor(\widehat{G_i})$ (see Corollary 4.7.9 in \cite{GZ}). 
If in addition $Aut(G_i)$ is dense in $Aut(\widehat G_i)$, then $k_i=k_{G_i,H}=1$. 
\end{rem}

\begin{rem} \label{ki finitos}
If $k_i=1$ for each $i=1,2$ then it follows from Theorem \ref{genus amalgam} that $g(G, \F)=g(G, \W)$. This happens for example when  $G_i\,(i=1,2)$ are finite or $H$ is characteristic in $\widehat G_i\,(i=1,2)$. 
\end{rem} 


\begin{rem} Note that if $G_i$ has only finitely many conjugacy classes of finite subgroups (as in any arithmetic group for example) then  $k_i$ is finite.  
\end{rem}

  


From Corollary \ref{dense estimate} we deduce then

\begin{cor}\label{genus estimate finite groups} If  $G_1, G_2$ are finite and $G=G_1*_H G_2$ is profinitely non-symmetric, 
then $g(G, \F) \leq  |\widetilde N_{ G_2}(H)\backslash \widetilde N_{G}(H)/\widetilde N_{G_1}(H)|.$ 
\end{cor} 

\begin{exa} Let $G_1=M\times H$ and $G_2=N\times H$ such that $M\not\cong N$ are residually finite finitely generated minimax soluble torsion free non-cyclic groups (for instance polycyclic) and $H$ is a finite  group.  Then $G_1,G_2$ are $\Oe$-groups (cf. Proposition \ref{virtually cyclic}), $\widehat G_1=\widehat M\times H$, $\widehat G_2=\widehat N\times H$ and $H$ is characteristic in $\widehat G_i$ for each $i=1,2$. So by Remark \ref{ki finitos} $k_i=1\,(i=1,2)$. Note  that $\overline{Aut}_{\widehat{G_1}}(H)=\overline{Aut}_{G_1}(H)=Aut(H)=\overline{Aut}_{\widehat G_2}(H)=\overline{Aut}_{G_2}(H)$ and $\widetilde N_{\widehat G_i}(H)=1\,(i=1,2),$ so $\widetilde{N}^{+} = \widetilde N_{\widehat G}(H)=1$ by Proposition \ref{cp1}. Note that $G=G_1*_{H} G_2$ is a non-symmetric  amalgamated free product. It follows from Theorem \ref{fixed subgroups}  that $g(G,\mathcal{F})=1$. 

\end{exa}

\begin{rem} \label{finite-by-cyclic genus} It follows from Remarks \ref{finite-by-cyclic isomorphism problems}, \ref{finite-by-cyclic profinite isomorphism problems2} and \ref{finite-by-cyclic profinite isomorphism problems} that the results of this  section  also hold if we assume that  $G_i$ ($i=1,2)$ are semidirect product $M_i\rtimes \Z$ such that $M_i\neq H$ and $|M_i|< \infty$ (respectively,  $M_i\rtimes \widehat{\Z}$ such that $M_i\neq H$ and $|M_i|< \infty$, in the profinite case).
\end{rem}

\begin{exa} Let $G_1=M\times H$ and $G_2=N\times H$ such that $M\cong N$ are residually finite finitely generated minimax soluble torsion free non-cyclic groups (for instance polycyclic) and $H$ is a finite  group.  Then $G_1,G_2$ are $\Oe$-groups (cf. Proposition \ref{virtually cyclic}), $\widehat G_1=\widehat M\times H$, $\widehat G_2=\widehat N\times H$ and $H$ is characteristic in $\widehat G_i$, $i=1,2$. So by Remark \ref{ki finitos} $k_i=1\,(i=1,2)$.  Note  that $\overline{Aut}_{\widehat{G_1}}(H)=\overline{Aut}_{G_1}(H)=Aut(H)=\overline{Aut}_{\widehat G_2}(H)=\overline{Aut}_{G_2}(H)$ and $\widetilde N_{\widehat G_i}(H)=1$, so $\widetilde{N}^{+} = \widetilde N_{\widehat G}(H)=1$ by Proposition \ref{cp1}. Note that $G=G_1*_{H,\mu} G_2$ is a symmetric amalgamated free product for any $\mu\in Aut(H)$. It follows from Theorem  \ref{case symmetrical genus} that $g(G,\mathcal{F})=1$. 

\end{exa}

\begin{exa} Let $P=\langle x_1, x_2\rangle$ be a  free nilpotent group of rank $2$ and $H\cong \Z/p \Z$ be  a group of odd prime order $p$. Note that $Aut(\Z/p \Z)=\langle c\rangle\cong C_{p-1}$ is a  cyclic group of order  $p-1$.   We look at $c$ as the unit of $\Z/p\Z$ and define   $G_i=H\rtimes_{\psi_i} P$ with $h^{x_j}= h^{c}$, for $i,j=1,2$ and note that they are $\Oe$-groups by Proposition \ref{virtually cyclic}.  
Note that $G=G_1*_H G_2$ is symmetric.  Since $H$ is characteristic in $\widehat G_i$, $k_1=1=k_2$ by Remark \ref{ki finitos}. Note that $\psi_i:P\longrightarrow Aut(H)\,(i=1,2)$ factors through the abelianization $P^{ab}$ of $P$ and denoting by $\bar\psi_i: P^{ab}\longrightarrow \psi_i(P)$ the corresponding natural factorization. Then we can consider the following semidirect products   $R_i=H\rtimes_{\bar\psi_i}P^{ab}\,(i=1,2)$. By \cite{A} $Aut(P)\longrightarrow Aut( P^{ab})$ is surjective and by  \cite{L} so is $Aut(\widehat P)\longrightarrow Aut(\widehat P^{ab})$. It follows that $\widetilde{Aut}_{ G_i}(H)=\widetilde{Aut}_{ R_i}(H)$.  By Lemma 2.1 \cite{Die}  $\overline{Aut}_{ R_i}(H)$ contains all automorphisms of $Aut(H)$ that lift to automorphisms of $R_i$ and in the proof of \cite[Lemma 3.20]{GZ} it is proved that every automorphism of $H$ lifts to an automorphism of $R_i,$ because $P$ is a quotient of a free group of rank $2.$ So we have $$\overline{Aut}_{ R_i}(H)=\overline{Aut}_{ G_i}(H)=\widetilde{Aut}_{ R_i}(H)=\widetilde{Aut}_{\widehat{G}_i}(H)=Aut(H)=Out(H).$$ It follows from Theorem  \ref{case symmetrical genus} that $g(G,\mathcal{F})=1$. 
\end{exa}


\begin{exa}\label{free by Cp} Let $H$ be a group of  order $3$ and $M$ a  $\Z H$-lattice of  $\Z$-rank $2$ with non-trivial action.   Let $G_1=M\rtimes_{\varphi} H$ be the corresponding semidirect product  and $D=G_1*_HG_1$ the double group. Then $G_1$ is $\Oe$-group (by Proposition \ref{virtually cyclic}) and has just two subgroups of order $3$ up to conjugation that can be swapped by an automorphism, so   $k_1=1=k_2$. 
Since $D$ is  a double, we  just need to calculate $\left|\left( \widetilde{Aut}_{G_1}(H)\backslash \widetilde{Aut}_{\widehat G_1}(H)/\widetilde{Aut}_{G_1}(H)\right)/C_2\right|$ by Theorem \ref{genus  double} and Remark \ref{ki finitos}. 
By Lemma 2.1 \cite{Die}  $\overline{Aut}_{ G_1}(H)$ contains  $\overline N_{Aut(M)}(H)$. Note that $Aut(M)\cong GL_2(\Z)$ that contains $S_3$ and a subgroup of order $3$ is unique up to conjugation in $GL_2(\Z)$. So $\widetilde{Aut}_{G_1}(H)=\widetilde{Aut}_{\widehat{G}_1}(H)=Aut(H)=Out(H)\cong C_2$.   
Thus by 
 Theorem \ref{genus double} $g(D,\F)=1$.
\end{exa}

\begin{exa} Let $G_1\cong GL_n(\f_p)\cong G_2$ such that $G_2=G_1^{op}$ is the opposite  group. Let $H$ be the Borel subgroup of the upper triangular matrices of $G_1$.  Note that the transpose map $\gamma: G_1 \longrightarrow G_2$ satisfying $\gamma(X)=X^T\,(X \in G_1)$  is an isomorphism. Let $H_2$ be the Borel subgroup of the lower triangular matrices of $G_2.$ Note that $\gamma(H)=H_2$. Define $\mu=\gamma|_H \in Iso(H, H_2)$, $G=G_1*_{H, \mu} G_2$ and identify $H$ with $H_2.$  We have that $G$ is symmetrical and by Remark \ref{ki finitos} $g(G, \F)=g(G, \W)$, furthermore, $H$ is self-normalizing in $G_1$ and $G_2$ so $\widetilde{N}_{G_1}(H)=\widetilde{N}_{G_2}(H)=\widetilde{N}_{G}(H)=1=\widetilde{N}^{+}$ (see Proposition \ref{cp1}). Therefore $\mathcal{K}=\widetilde{Aut}_{G_1}(H)$ and as $\xi=id$ we have $\mathcal{K}=\mathcal{S}=\widetilde{Aut}_{G_1}(H).$ It follows from  Theorem  \ref{case symmetrical genus} that $g(G,\W)=g(G, \F)=1$. If $D=G_1 *_{H} G_1$ is the double then by Corollary \ref{genero 1 finito} $g(D, \F)=1$.

\end{exa}



\begin{rem}\label{class A} Let $\A$ be the class of all accessible groups whose vertex groups of its JSJ-decomposition are $\widehat{OE}$-groups. Let $G=G_1*_H G_2$ a free product of $\Oe$ groups  with finite amalgamation. Let $B$ be a group in $\A$ such that $\widehat G\cong \widehat B$. Then
 by \cite[Theorem 1.1]{BPZ} $B=B_1*_K B_2$ with $\widehat G_i\cong \widehat B_i$, $i=1,2$ and $H\cong K$. Thus  $$g(G,\A)\leq \sum_{B_1\in \mathfrak g(G_1,\A),B_2\in \mathfrak g(G_2,\A)} g(B,\F),$$ where $B_1$ and $B_2$ ranges over representatives of isomorphism classes in $\mathfrak g(G_1,\A)$, $\mathfrak g(G_2,\A)$ respectively. 
 \end{rem}


\nocite{*}

\begin{thebibliography}{99}


\bibitem{Aka} M. Aka. {\it Arithmetic groups with isomorphic finite quotients.} Journal of Algebra, 352 (2012) 322-240.

\bibitem{A} S. Andreadakis. {\it On the automorphisms of free groups and free nilpotent groups.} Proc. London Math. Soc. 15 (1965) 239-268.

\bibitem{BAU74} G. Baumslag. {\it Residually finite groups with the same finite images.} Compositio Mathematica, 29 (\textbf{3}) (1974) 249-252.



\bibitem{BCR16}  M.R. Bridson, M.D.E. Conder, A.W. Reid. {\it Determining Fuchsian groups by their finite quotients.} Israel Journal of Mathematics, \textbf{214} (2016) 1-41.
	
\bibitem{BMRS18} M.R. Bridson, D.B. McReynolds, A.W. Reid, R. Spitler. \textit{ Absolute profinite rigidity and hyperbolic geometry}. Annals of Mathematics,  {\bf 192} (2020) 679--719. arXiv preprintarXiv:  1811.04394, 2018.
	
\bibitem{BMRS20} M.R. Bridson, D.B. McReynolds, A.W. Reid, R. Spitler. \textit{On the profinite rigidity of triangle groups}, Bulletin of the London Math. Society, {\bf 53} (2021) 1849--1862. arXiv: 2004.07137, 2020.

\bibitem{BPZ}  V.R. de Bessa, A.L.P. Porto, P.A. Zalesskii. {\it The profinite completion of accessible groups.} Monatsh. Math. (2022). https://doi.org/10.1007/s00605-022-01789-9, https://arxiv.org/submit/4467165.

\bibitem{BZ} V.R. Bessa, P.A. Zalesskii. {\it The genus for HNN-extensions.} Mathematische Nachrichten, v. 286, (2013) 817-831.





\bibitem{DI} W. Dicks, M.J. Dunwoody. Groups acting on graphs. Cambridge Studies in Advanced Math. {\bf 17}, Cambridge University Press, 1989.

\bibitem{Die} J. Dietz. {\it Automorphism of groups of semi-direct products.}  Communications in Algebra {\bf 40} (2012) 3308-3316.  

\bibitem{GPS} F.J. Grunewald, P.F. Pickel, D. Segal. {\it Finiteness theorems for polycyclic groups.} Bull. Amer. Math. Soc. (N.S.) 1 (1979), no. 3, 575-578.

\bibitem{GS} F. Grunewald, R. Scharlau. {\it A note on finitely generated torsion-free nilpotent groups of class 2.} J. Algebra 58 (\textbf{1}) (1979) 162-175.

\bibitem{GZ} F.J. Grunewald, P.A. Zalesskii. {\it Genus for groups.} Journal of Algebra 326, (2011) 130-168.

\bibitem{J} A. Jaikin-Zapirain, The finite and solvable genus of finitely generated free and surface groups. Preprint. https://matematicas.uam.es/~andrei.jaikin/preprints/profinitefree.pdf




\bibitem{L} A. Lubotzky. {\it Combinatorial group theory for pro-p-groups.} J. Pure and Appl. Algebra 25 (1982) 311-325.


\bibitem{LS} R.C. Lyndon, P.E. Schupp. Combinatorial group theory. Reprint of the (1977) edition. Classics in Mathematics. Springer-Verlag, Berlin, 2001.

\bibitem{MKS} W. Magnus, A. Karrass, D. Solitar. Combinatorial group theory: Presentations of groups in terms of generators and relations. Second edition, Dover Publication, INC. New York, 1976.

\bibitem{M} I.  Morales, On the profinite rigidity of free and surface groups. arxiv 2211.12390.


\bibitem{Ner19} G.J. Nery, \textit{Profinite genus of fundamental groups of torus bundles.} Communications in Algebra, \textbf{48} (2019) 1567-1576.

\bibitem{Ner20} G.J. Nery. {\it Profinite genus of fundamental groups of compact flat manifolds with holonomy group of prime order.} Journal of Group Theory, \textbf{24\,(6)}, 1135-1148. 






\bibitem{R} L. Ribes. Profinite graphs and groups. Springer-Verlag {\bf 66}, 2017.

\bibitem{RO} D.J.S. Robinson. A course on the theory of groups. Springer-Verlag, New York, Berlin-Heidelberg, 1982.

\bibitem{RSZ} L. Ribes, D. Segal, P.A. Zallesskii. {\it Conjugacy separability and free products of groups with cyclic amalgamation.} J. London Math. Soc. 57 (\textbf{2}), (1998) 609-628.

\bibitem{RZ} L. Ribes, P.A. Zalesskii. Profinite groups. Second edition, Springer-Verlag, Berlin Heidelberg, 2010.


\bibitem{S} J.P. Serre. Trees. Springer-Verlag, 2003.

\bibitem{Ste72} P.F. Stebe. {\it Conjugacy separability of groups of integer matrices.} Proceedings of the American Mathematical Society, 32 (\textbf{1}) (1972)  1-7.


\bibitem{Wil17} G. Wilkes. \textit{Profinite rigidity for Seifert fibre spaces.} Geometriae Dedicata, \textbf{188} (2017) 141-163.


\bibitem{ZM} P.A. Zalesskii, O.V. Melnikov. {\it Fundamental Groups of Graphs of Profinite Groups.} {\empty Algebra i Analiz} {\bf 1} (1989); translated in: {\empty Leningrad Math. J.} {\bf 1} (1990) 921-940.

\bibitem{ZM1} P.A. Zalesskii, O.V. Melnikov. {\it Subgroups of profinite groups acting on trees.} {\empty Math. USSR Sbornik} {\bf 63} (1989) 405-424.

\end{thebibliography}
\end{document}